\documentclass{article}
\usepackage{amsmath,amsthm, amssymb}
\usepackage{graphicx, color}               
\usepackage{tikz,caption}
\usepackage{epstopdf}
\usepackage{mathtools,hyperref,verbatim}
\usepackage[numbers]{natbib}
\usepackage[english]{babel}
\newcommand{\bit}{\begin{itemize}}
\newcommand{\eit}{\end{itemize}}

\newcommand{\ch}{\text{ch}}

\newcounter{ctr}

\newtheorem{thm}[ctr]{Theorem}

\newtheorem{lemma}[ctr]{Lemma}
\newtheorem{conj}[ctr]{Conjecture}

\title{Equitable Choosability of Prism Graphs}
\author{Kirsten Hogenson \thanks{Skidmore College, {khogenso@skidmore.edu}} \and Dan Johnston \thanks{Trinity College, daniel.johnston@trincoll.edu} \and Suzanne O'Hara \thanks{Wesleyan University, seohara@wesleyan.edu}}
\date{}

\begin{document}

\maketitle

\begin{abstract}
    \noindent A graph $G$ is equitably $k$-choosable if, for every $k$-uniform list assignment $L$, $G$ is $L$-colorable and each color appears on at most $\left\lceil |V(G)|/k\right\rceil$ vertices.  Equitable list-coloring was introduced by Kostochka, Pelsmajer, and West in 2003 \cite{KPW}.  They conjectured that a connected graph $G$ with $\Delta(G)\geq 3$ is equitably $\Delta(G)$-choosable, as long as $G$ is not complete or $K_{d,d}$ for odd $d$.  In this paper, we use a discharging argument to prove their conjecture for the infinite family of prism graphs.  \\
    
\end{abstract}


\section{Introduction}
\label{sec:intro}

For terminology and notation not explicitly defined in this paper, see \cite{Diestel}.  Let $G$ be a graph.  We say that $G$ is \textbf{properly $k$-colorable} if there exists a mapping $c:V(G) \to [k]$ such that $c(u) \ne c(v)$ for every edge $uv \in E(G)$.  In some graph coloring scenarios, we impose restrictions on which colors it is acceptable to use on each vertex.  A \textbf{list assignment} $L$ for $G$ provides a list of acceptable colors, $L(v)$, to each vertex $v\in V(G)$.  This list assignment is said to be \textbf{$k$-uniform} if $|L(v)|=k$ for every $v\in V(G)$.  A proper $L$-coloring of $G$ is a proper coloring in which each vertex is assigned a color from its list.  If every $k$-uniform list assignment $L$ of $G$ admits a proper list coloring, we say that $G$ is \textbf{$k$-choosable}.  The choice number, $\ch(G)$, of a graph $G$ is the smallest $k$ for which $G$ is $k$-choosable.  The concept of list coloring was independently introduced by Vizing\cite{vizing} and Erd\H{o}s, Rubin, and Taylor\cite{ERT}, and the latter authors were able to bound the choice number in the following way.
\begin{thm}[Erd\H{o}s, Rubin, and Taylor \cite{ERT}]\label{thm:ert}
    If a connected graph $G$ is not $K_n$ and not an odd cycle, then $\ch(G)\leq \Delta(G)$.
\end{thm}

A vertex coloring of $G$ partitions $V(G)$ into sets called \textbf{color classes}.  All vertices in a particular color class are assigned the same color, so if the coloring is proper, each color class is an independent set of vertices. Depending on the context, we may wish to ensure that the color classes are roughly uniform in size.  A proper $k$-coloring is said to be \textbf{equitable} if the size of every color class is either $\lceil |V(G)|/k\rceil$ or $\lfloor |V(G)|/k \rfloor$.  In \cite{KPW}, Kostochka, Pelsmajer, and West introduced a variation of equitable coloring for $k$-uniform list colorings.  In particular, they say that a graph $G$ is \textbf{equitably $k$-choosable} if every $k$-uniform list assignment of $G$ admits a proper coloring in which each color class is \textbf{$\lceil |V(G)|/k\rceil$-bounded}.  That is, they allow each color class to have any size up to $\lceil |V(G)|/k\rceil$.

In this paper, we will investigate the choosability and equitable choosability parameters of the infinite family of \textbf{prism graphs}, $\{\Pi_n ~:~ n\geq 3\}$.  Note that the $n$-prism $\Pi_n$ is the Cartesian product $C_n \square K_2$, whose drawing resembles a geometric prism with an $n$-sided polygon base.  We will refer to the $n$ edges which attach the two copies of $C_n$ in $\Pi_n$ as \textbf{rungs}.  Note that the prism graphs are cubic, Hamiltonian, and planar (but not outerplanar).  When $n$ is even, the prism $\Pi_n$ is bipartite.  And when $n\geq 4$, the prism $\Pi_n$ has girth 4.

There are several known results which are relevant to our research.  Gr\"otzsch proved that planar graphs of girth at least 4 are 3-colorable \cite{grotzsch}, and Alon and Tarsi proved that bipartite planar graphs are 3-choosable \cite{alontarsi}.  These results, along with Theorem~\ref{thm:ert}, suggest that the choice number of all prism graphs is likely 3.

There are also several relevant results known about equitable choosability.  Zhu and Bu proved that all 3-regular outerplanar graphs are equitably 3-choosable \cite{ZhuBu}.  Dong and Zhang proved that a graph $G$ with maximum average degree less than 3 is equitably $k$-colorable and equitably $k$-choosable for $k\geq \max\{\Delta(G),4\}$ \cite{DZ}. Pelsmajer proved that graphs with maximum degree equal to 3 are equitably $k$-choosable for $k>3$ \cite{pelsmajer}, and Wang and Lih proved the following related theorem about graphs with maximum degree at most 3.
\begin{thm}[Wang and Lih \cite{WangLih}]\label{thm:wanglih}
    Every graph $G$ with $\Delta(G)\leq 3$ is equitably $k$-choosable whenever $k>\Delta(G)$.
\end{thm}

In 1994, Chen, Lih, and Wu proposed the Equitable $\Delta$-Coloring Conjecture:
\begin{conj}[The Equitable $\Delta$-Coloring Conjecture\cite{ChenLihWu}]
    Let $G$ be a connected graph.  If $G$ is not a complete graph, or an odd cycle, or a complete bipartite graph $K_{2m+1,2m+1}$, then $G$ is equitably $\Delta(G)$-colorable.
\end{conj}
They also proved the following special case of their conjecture:
\begin{thm}[Chen, Lih, and Wu \cite{ChenLihWu}]\label{thm:chenlihwu}
    A connected graph $G$ with $\Delta(G)\leq 3$ is equitably $\Delta(G)$-colorable if it is different from $K_m$, $C_{2m+1}$, and $K_{2m+1,2m+1}$ for all $m\geq 1$.
\end{thm}

In 2003, when Kostochka, Pelsmajer, and West introduced their list analogue of equitable coloring, they also included a list analogue of the Equitable $\Delta$-Coloring Conjecture:
\begin{conj}[Kostochka, Pelsmajer, and West \cite{KPW}]\label{conj:kpw}
    If $G$ is a connected graph with maximum degree at least 3, then $G$ is equitably $\Delta(G)$-choosable, unless $G$ is a complete graph or is $K_{d,d}$ for some odd $d$.
\end{conj}

Our main result verifies that Conjecture~\ref{conj:kpw} holds for the infinite family of prism graphs.
\begin{thm}\label{thm:mainthm}
    $\Pi_n$, $n\geq 3$, is equitably 3-choosable.
\end{thm}
Our proof of Theorem~\ref{thm:mainthm} is organized as follows:
In Section~\ref{sec:3chooseprisms}, we start by verifying that the choice number of all prisms is 3.  Then in Section~\ref{sec:natmost5}, we show that $\Pi_n$ is equitably 3-choosable for $n\in\{3,4,5\}$.  In Section~\ref{sec:natleast6}, we complete the proof that all prisms are equitably 3-choosable.  This argument is broken into two subsections.  In Section~\ref{subsec:bluereddifferby0or1} we prove that in a minimum list coloring of $\Pi_n$, $n\geq 6$, the two largest color classes must differ in size by at least 2 vertices.  Finally, in Section~\ref{subsec:blue2bigger} we use a discharging argument to prove that a minimum list coloring of $\Pi_n$, $n\geq 6$, must be equitable.


\section{3-Choosability of Prisms}
\label{sec:3chooseprisms}

For a graph to be equitably $k$-choosable, it must be $k$-choosable.  Thus, we begin by proving that the prism graphs are not only $3$-choosable, but have choice number equal to 3. This result was proved by the third author as part of their 2021 undergraduate thesis, and their proof is recreated here.

\begin{thm}\label{thm:3choose}
    $\ch(\Pi_n)=3$ for all $n\geq 3$.
\end{thm}
\begin{proof}
    A result of Erd\H{o}s, Rubin, and Taylor gives the upper bound.  First, notice that the prism graphs are 3-regular, and thus $\Delta(\Pi_n)=3$.  So by Theorem~\ref{thm:ert}, we have $\ch(\Pi_n) \leq 3$.
    
    To prove the lower bound, we will provide examples of 2-list-assignments of $\Pi_n$ which do not admit proper colorings.  There are two cases to consider.
    
    If $n$ is odd, let $L$ assign to every vertex of $\Pi_n$ the list $\{r,b\}$. A proper $L$-coloring of $\Pi_n$ would then corresponding to a proper 2-coloring of $\Pi_n$.  But since $n$ is odd, $\Pi_n$ contains an odd cycle, which cannot be properly 2-colored.  Thus, $L$ does not admit a proper coloring of $\Pi_n$, and $\ch(\Pi_n)\geq 3$ when $n$ is odd.
    
    If $n$ is even, then $\Pi_n$ contains a subgraph consisting of 4 consecutive rungs.  We let $L$ assign lists to that subgraph as shown in Figure~\ref{fig:pinnot} and attempt to properly color it in the following way: Notice that if vertex $v$ is assigned color $b$, then we are unable to properly color vertex $u$ as shown in the left diagram.
    Similarly, if vertex $v$ is assigned color $g$, then we are unable to properly color vertex $u$ as shown in the right diagram.
    Because this list assignment does not allow us to properly color the pictured subgraph, it does not allow us to properly color the full graph.  Thus, $\ch(\Pi_n)\geq 3$ when $n$ is even, and our proof is complete.
    
    \begin{figure}[ht!]
    \centering
    \begin{minipage}{0.4\textwidth}
    \includegraphics[width=\linewidth]{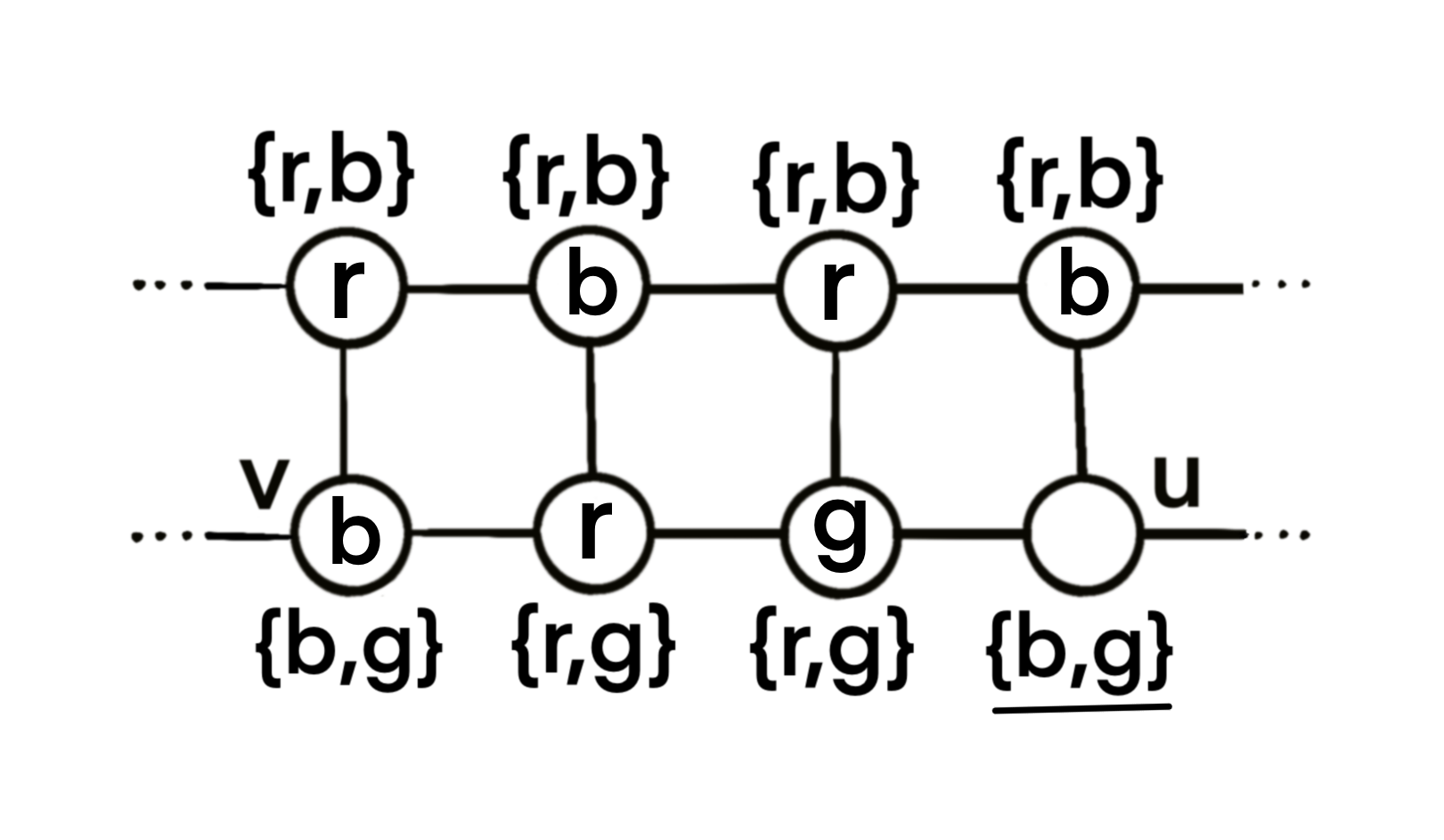}
    \end{minipage}\hspace{0.05\textwidth}
    \begin{minipage}{0.4\textwidth}
    \includegraphics[width=\linewidth]{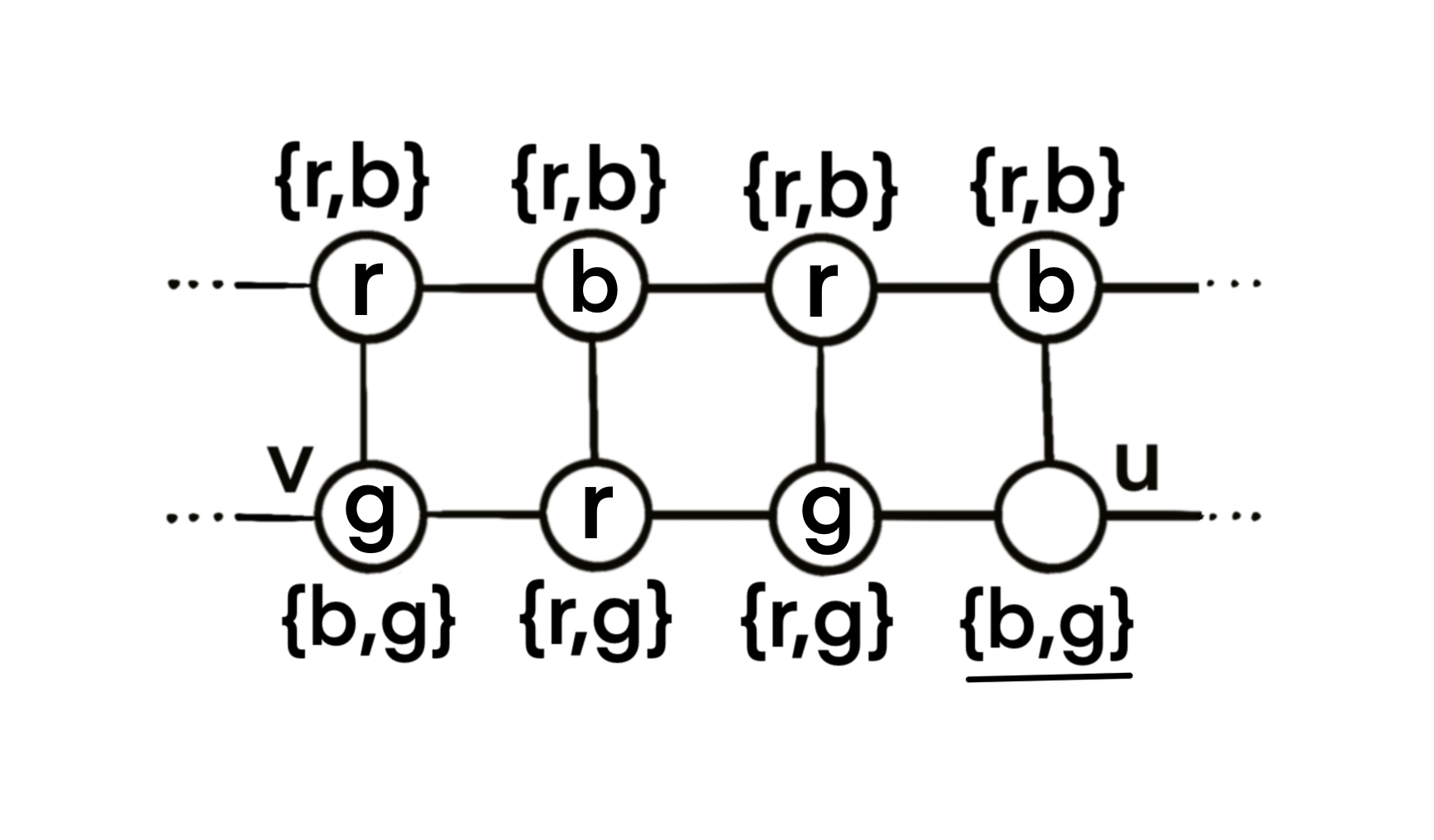} 
    \end{minipage}
    \caption{Failed colorings of $\Pi_n$}
    \label{fig:pinnot}
    \end{figure}
\end{proof}

One consequence of Theorem~\ref{thm:3choose} is that the prisms are \emph{not} equitably 1- or 2-choosable because they are neither 1- nor 2-choosable.  Also, by Theorem~\ref{thm:wanglih}, the prisms \emph{are} equitably $k$-choosable for $k\geq 4$.  So the question that remains is whether or not the prisms are equitably 3-choosable. In the remaining sections of this paper, we will answer this question in the affirmative.


\section{Small Prisms}
\label{sec:natmost5}

In this section we prove that $\Pi_3$, $\Pi_4$, and $\Pi_5$ are equitably 3-choosable.  The proofs for $\Pi_4$ and $\Pi_5$ were done by the third author in their 2021 undergraduate thesis and are recreated below.  For the proof of Lemma~\ref{lem:35prisms}, recall that the \emph{independence number}, $\alpha(G)$, of a graph $G$ is the cardinality of a maximum set of pairwise nonadjacent vertices of $G$.

\begin{lemma}\label{lem:35prisms}
    $\Pi_3$ and $\Pi_5$ are equitably 3-choosable.
\end{lemma}
\begin{proof}
    First, recall that $\Pi_3$ and $\Pi_5$ are both 3-choosable by Theorem~\ref{thm:3choose}, so every 3-list-assignment of these graphs admits a proper coloring.
    
    For $\Pi_3$ to be equitably 3-choosable, every 3-list-assignment of $\Pi_3$ must admit a proper coloring in which the cardinality of a largest color class is at most $\left\lceil 6/3 \right\rceil =2$.  Since each color class in a proper coloring is an independent set, and the independence number of $\Pi_3$ is $\alpha(\Pi_3)=2$, it must be that every color class contains at most 2 vertices.  Thus, $\Pi_3$ is equitably 3-choosable.

    Similarly, for $\Pi_5$ to be equitably 3-choosable, the cardinality of a largest color class must be at most $\left\lceil 10/3 \right\rceil = 4$.  The independence number of $\Pi_5$ is $\alpha(\Pi_5)=4$, so every color class contains at most 4 vertices, and $\Pi_5$ is equitably 3-choosable.
\end{proof}

To prove that $\Pi_4$ is equitably 3-choosable, it is not sufficient to simply calculate the independence number.  Instead, we perform a case analysis to show that every coloring which is not $\left\lceil |V(\Pi_4)|/3\right\rceil$-bounded can be transformed into a coloring which is.

\begin{lemma}\label{lem:4prisms}
    $\Pi_4$ is equitably 3-choosable.
\end{lemma}
\begin{proof}
    First, recall that $\ch(\Pi_4)=3$ by Theorem~\ref{thm:3choose}, so every 3-list-assignment of this graph admits a proper coloring.  Let $L$ be an arbitrary 3-list-assignment of $\Pi_4$ and let $c:V\to\mathcal{C}$ be a proper $L$-coloring of $\Pi_4$, where $\mathcal{C}$ is a set of at least three colors.  Without a loss of generality, suppose a largest color class of $c$ corresponds to the color $b$.  We will denote this color class Blue.
    
    If $|\text{Blue}|\leq \left\lceil 8/3 \right\rceil = 3$, then $c$ is 3-bounded and we have found the coloring we need.  So suppose instead that $|\text{Blue}|>3$.  The independence number of $\Pi_4$ is $\alpha(\Pi_4)=4$, so it must be that $|\text{Blue}|=4$.  Thus, due to the symmetries of $\Pi_4$, the coloring $c$ must look like one of the configurations in Figure~\ref{fig:pi4colorings}.
    Note that these cases are distinguished by the different possible colorings of $u_1$, $v_2$, $u_3$, and $v_4$, which must all have colors other than $b$.
    
    \begin{figure}[!ht]
    \begin{center}
    \begin{minipage}{0.3\textwidth}
    \includegraphics[width=\linewidth]{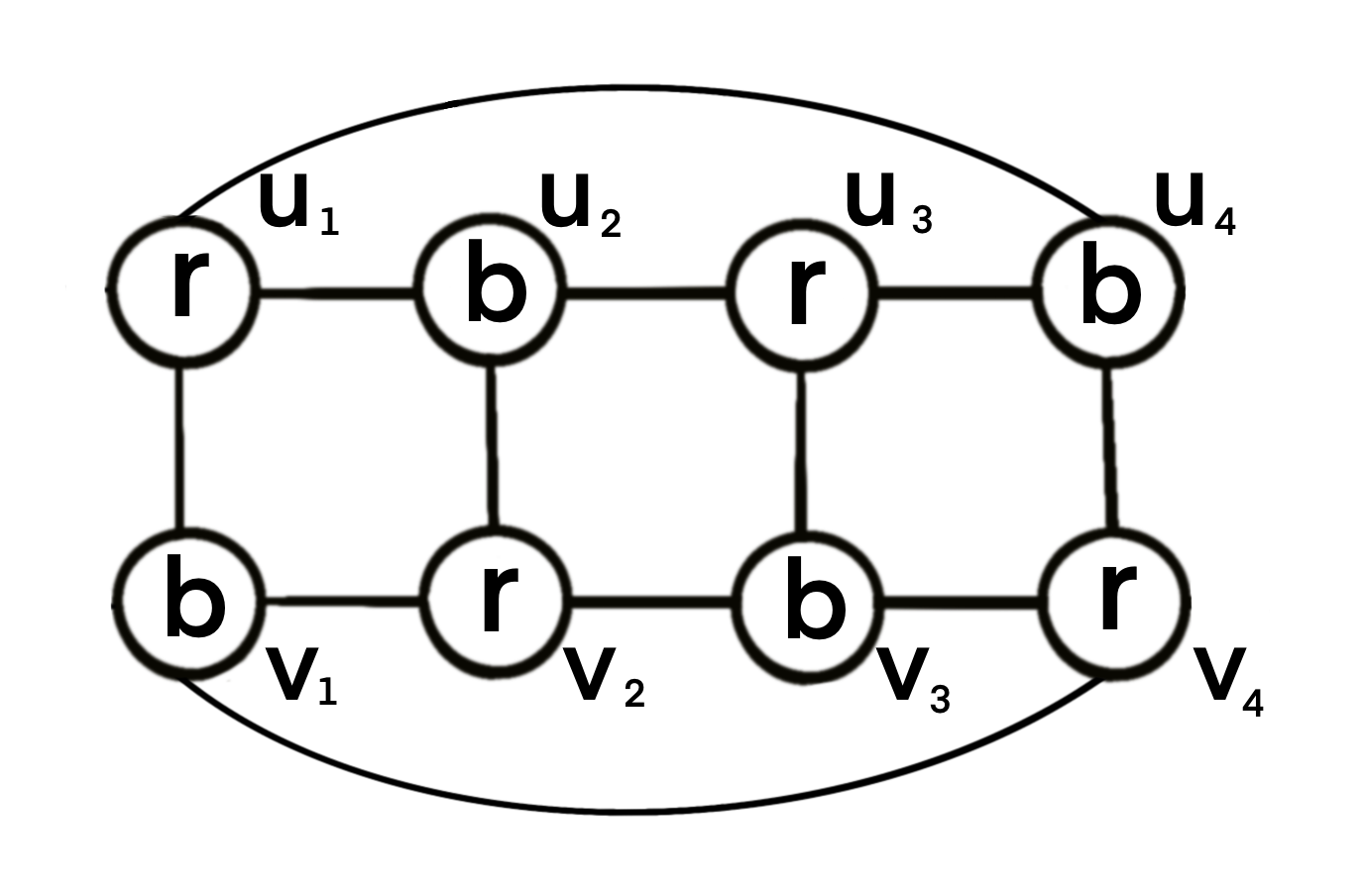}
    \captionof*{figure}{Case 1}
    \end{minipage}\hfill
    \begin{minipage}{0.3\textwidth}
    \includegraphics[width=\linewidth]{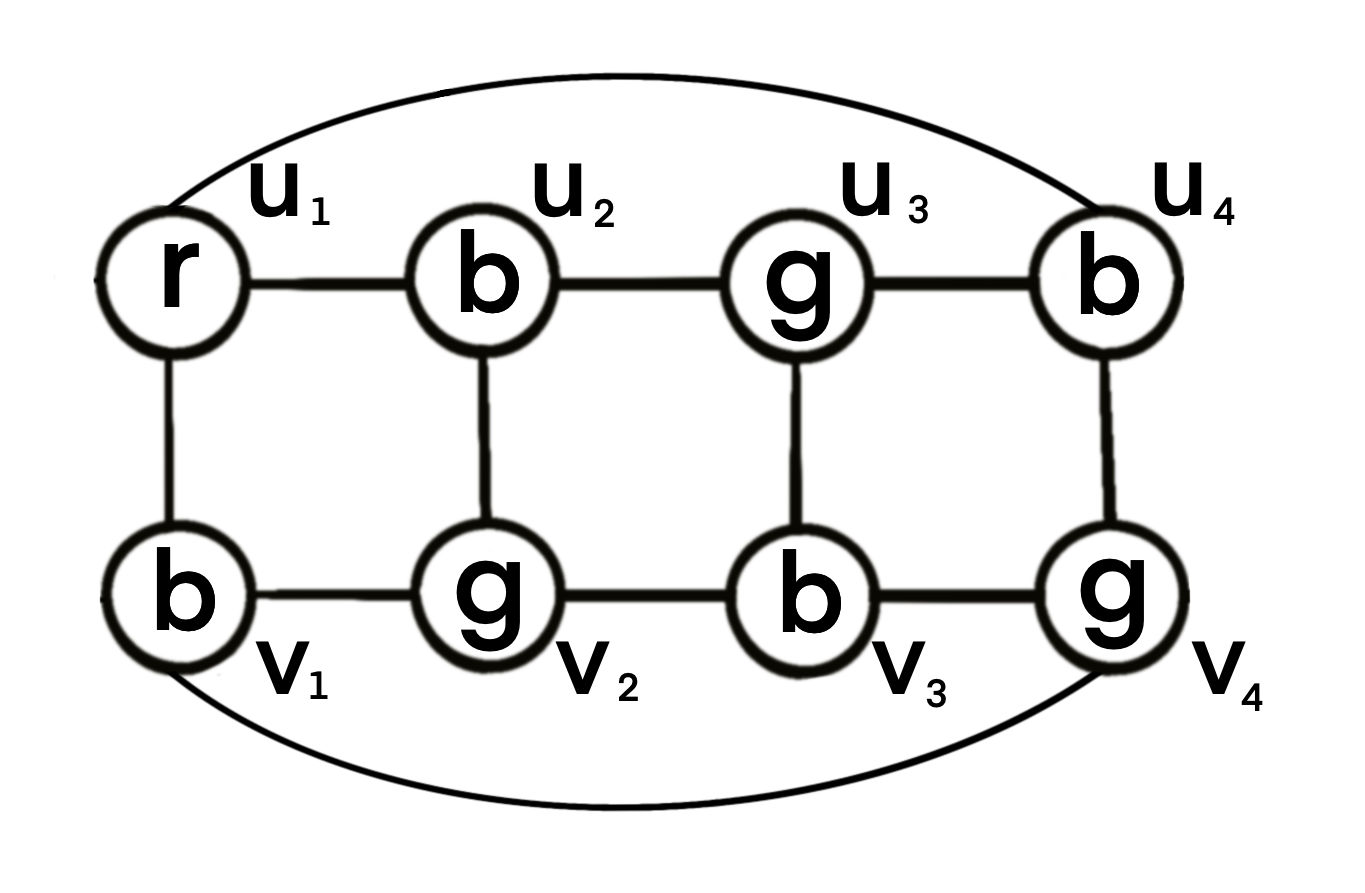} 
    \captionof*{figure}{Case 2}   
    \end{minipage}\hfill
    \begin{minipage}{0.3\textwidth}
    \includegraphics[width=\linewidth]{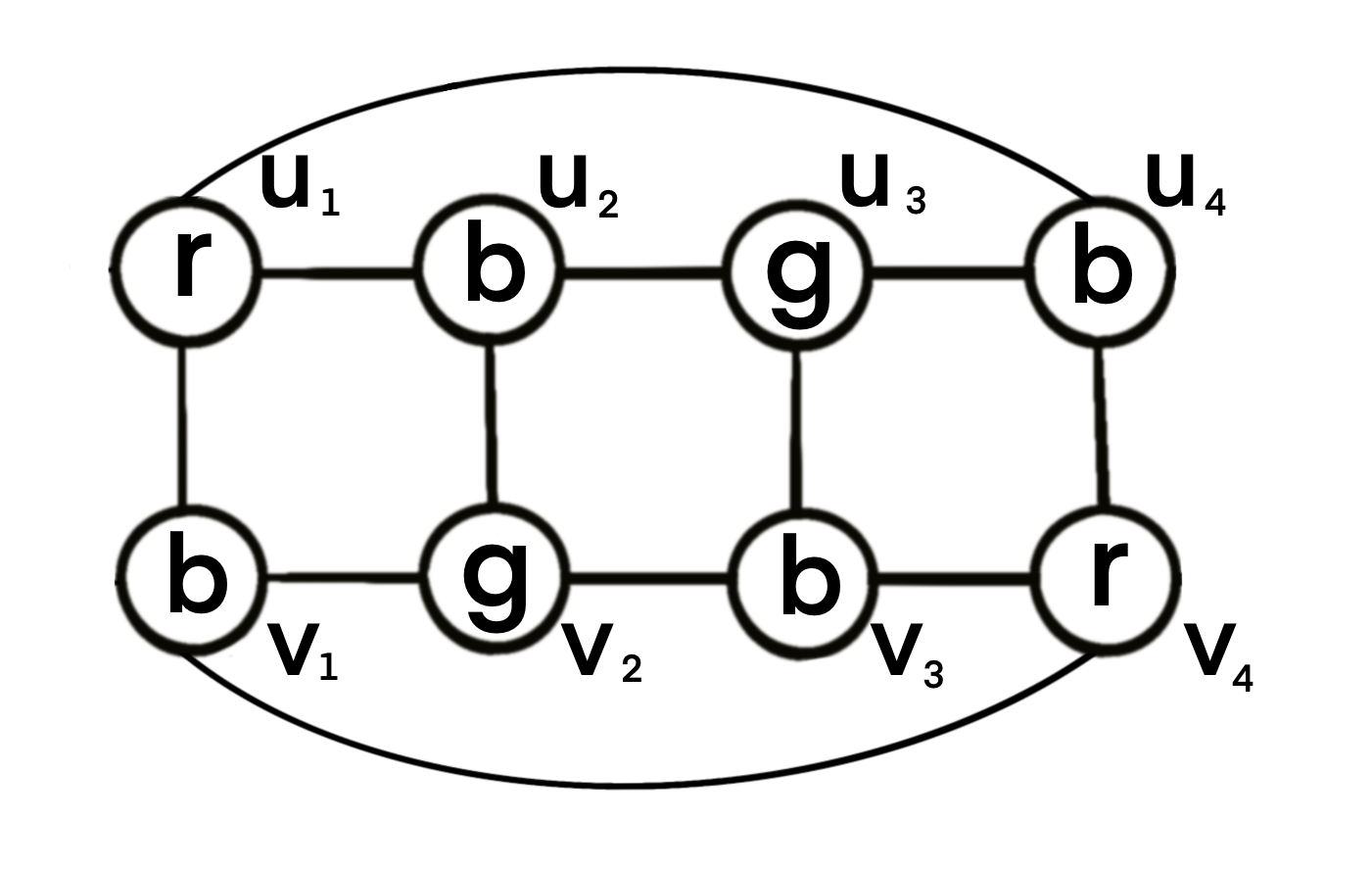} 
    \captionof*{figure}{Case 3}   
    \end{minipage} \\
    \begin{minipage}{0.3\textwidth}
    \includegraphics[width=\linewidth]{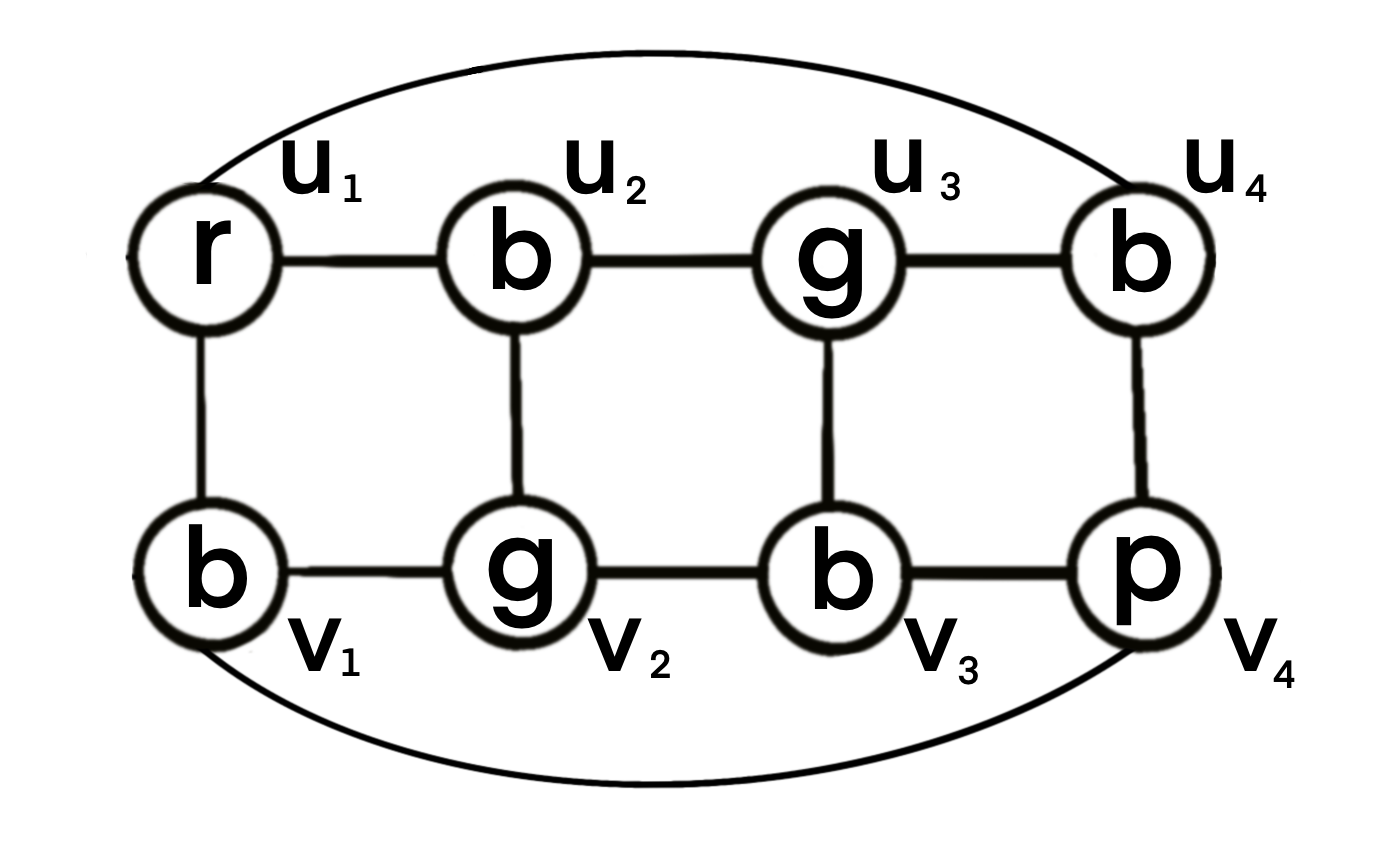} 
    \captionof*{figure}{Case 4}   
    \end{minipage}\hspace{0.1\textwidth}
    \begin{minipage}{0.3\textwidth}
    \includegraphics[width=\linewidth]{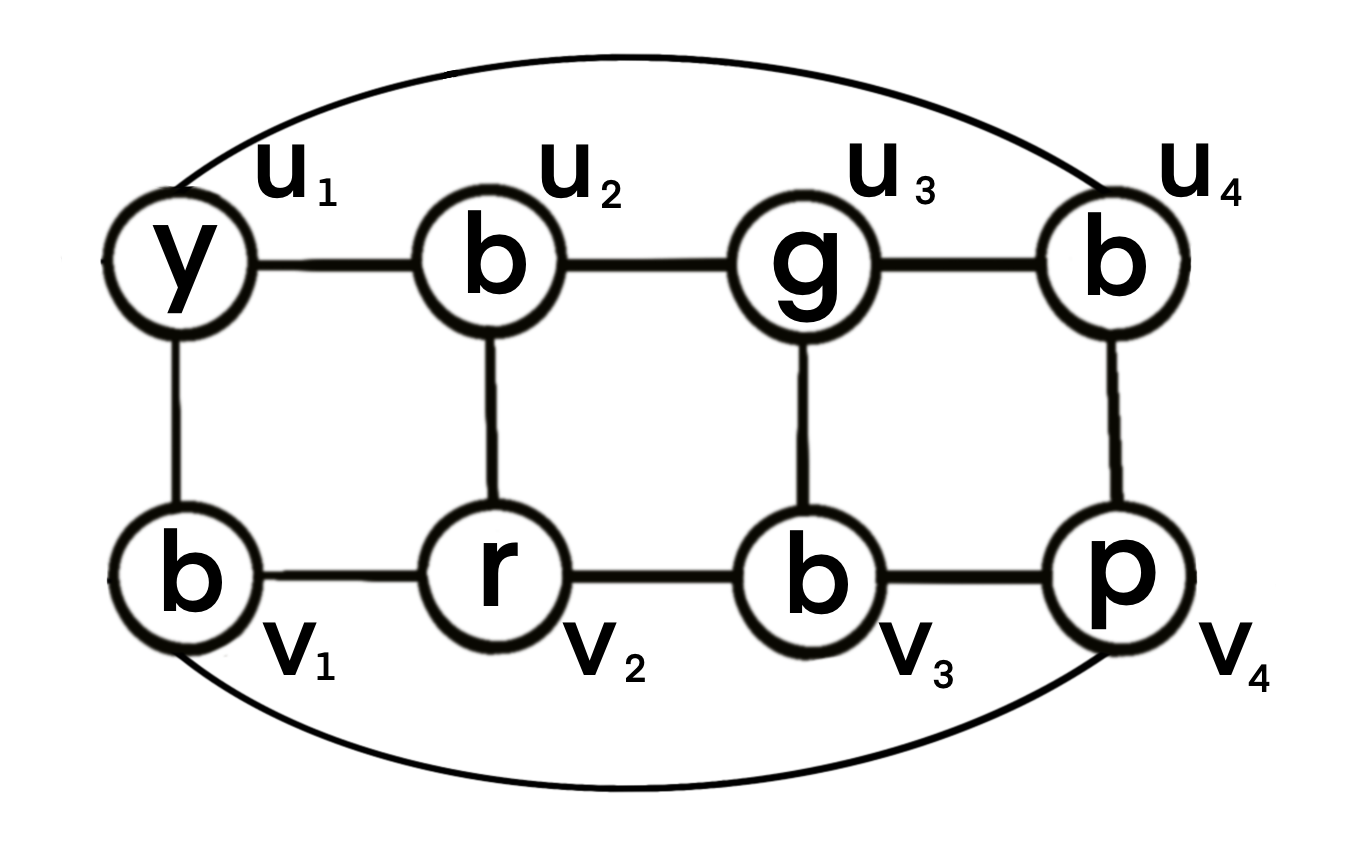} 
    \captionof*{figure}{Case 5}   
    \end{minipage}
\caption{Potential colorings of $\Pi_4$.}
\label{fig:pi4colorings}
\end{center}
\end{figure}
    
    In Case 1 there are two color classes of cardinality 4.  However, lists $L(u_1)$ and $L(v_3)$ must each contain a color which is neither $r$ nor $b$.  (This color may or may not be the same in the two lists.)  Define a new coloring $c'$ in which $c'(u_1), c'(v_3) \not\in\{b,r\}$ and $c'=c$ for all other vertices.  This new coloring $c'$ is a proper $L$-coloring whose largest color class has cardinality 3, so $L$ admits a proper 3-bounded coloring.
    
    In Case 2, list $L(v_3)$ must contain a color which is neither $b$ nor $g$.  Define a new coloring $c'$ in which $c'(v_3)\not\in\{b,g\}$ and $c'=c$ for all other vertices.  This new coloring $c'$ is a proper $L$-coloring whose largest color class has cardinality 3, so $L$ admits a proper 3-bounded coloring.
    
    In Cases 3 and 4, we will attempt to change the color of vertex $u_2$.  If there is a color in $L(u_2)$ which is not $b$, $g$, or $r$, we can obtain a proper 3-bounded $L$-coloring by changing $u_2$ to that color.  If not, then $L(u_2)=\{b,g,r\}$. Fortunately, list $L(u_1)$ must contain a color that is not blue or red.  Define a new coloring $c'$ in which $c'(u_2)=r$, $c'(u_1)\not\in\{b,r\}$, and $c'=c$ for all other vertices.  This new coloring $c'$ is a proper $L$-coloring whose largest color class has cardinality 3.  Thus, in Cases 3 and 4, $L$ always admits a proper 3-bounded coloring.
    
    In Case 5, we will attempt to change the color of vertex $u_2$.  If there is a color in $L(u_2)$ which is not $b$, $r$, $g$, or $y$, we can obtain a proper 3-bounded $L$-coloring by changing $u_2$ to that color.  If not, then $L(u_2)$ must contain $b$ and exactly two of the colors $r$, $g$, and $y$.  Without a loss of generality, suppose $L(u_2)$ contains $y$.  Fortunately, list $L(u_1)$ must contain a color that is not $b$ or $y$.  Define a new coloring $c'$ in which $c'(u_2)=y$, $c'(u_1)\not\in\{b,y\}$, and $c'=c$ for all other vertices.  This new coloring $c'$ is a proper $L$-coloring whose largest color class has cardinality 3.  Thus, in Case 5, $L$ always admits a proper 3-bounded coloring.
    
    In every case, we are able to obtain a proper 3-bounded $L$-coloring.  This means that the arbitrary 3-list-assignment $L$ always admits a proper $3$-bounded coloring, so by definition $\Pi_4$ is equitably 3-choosable.
\end{proof}


\section{Proof of Main Result}
\label{sec:natleast6}

In this section we prove that $\Pi_n$, $n\geq 6$, is equitably 3-choosable.  Together with Lemmas~\ref{lem:35prisms} and~\ref{lem:4prisms}, this constitutes a proof of our main result, Theorem~\ref{thm:mainthm}.  For the remainder of the paper, let $L$ be an arbitrary 3-list-assignment of $\Pi_n$.  We know that a proper $L$-coloring of $\Pi_n$ exists by Theorem~\ref{thm:3choose}, so we'll use a minimum counterexample technique and a discharging argument to prove that $L$ admits a proper $\left\lceil 2n/3\right\rceil$-bounded coloring.

Let $c$ be a proper $L$-coloring of $\Pi_n$, and let $C_1,C_2,\dots,C_r$ be the color classes of $c$.  Further, let $|C_i|=n_i$ for each color class and, without a loss of generality, suppose $n_1 \geq n_2 \geq \cdots \geq n_r$.  The \textbf{color word} of $c$ is the length-$r$ list $w_c := n_1 n_2 \cdots n_r$.  We say $c$ is a \textbf{lex-min coloring} of $\Pi_n$ if its color word $w_c$ is lexicographically minimum.  That is, if $c'$ is another proper $L$-coloring of $\Pi_n$ with color word $w_{c'}=m_1m_2\cdots m_s$, then either $w_{c}=w_{c'}$ or there exists an index $k$ ($1\leq k\leq s$), such that $n_i=m_i$ for $1\leq i \leq k-1$ and $n_k < m_k$.

The following two lemmas describe some useful properties of a lex-min $L$-coloring of $\Pi_n$, $n\geq 6$.

\begin{lemma}\label{lem:difflistsinnoneqlexmin}
    Let $c$ be a proper lex-min $L$-coloring of $\Pi_n$, $n\geq 6$. If the lists in $L$ are all the same, then $c$ is $\left\lceil 2n/3\right\rceil$-bounded. 
\end{lemma}
\begin{proof}
    Suppose $c$ is a proper lex-min $L$-coloring of $\Pi_n$.  If all lists of $L$ are the same, then finding a proper $L$-coloring of $\Pi_n$ is equivalent to finding a proper 3-coloring of $\Pi_n$.
    According to Theorem~\ref{thm:chenlihwu}, $\Pi_n$ is equitably 3-colorable.  That is, there exists a proper 3-coloring of $\Pi_n$ in which each color class has cardinality equal to $\left\lceil 2n/3\right\rceil$ or $\left\lfloor 2n/3\right\rfloor$.  In the lex-min coloring $c$, the largest color class may be even smaller than $\left\lceil 2n/3\right\rceil$.  Therefore, $c$ is $\left\lceil 2n/3\right\rceil$-bounded.
\end{proof}

\begin{lemma}\label{lem:cuses4colors}
    Let $c$ be a proper lex-min $L$-coloring of $\Pi_n$, $n\geq 6$. If $c$ is not $\left\lceil 2n/3\right\rceil$-bounded, then the range of $c$ includes at least 4 colors.
\end{lemma}
\begin{proof}
    By Lemma~\ref{lem:difflistsinnoneqlexmin}, we know that the lists used to define $c$ are not all the same, so at least 4 colors appear in the union of all the color lists.  Suppose $c$ uses fewer than 4 colors.  
    Then there exists a vertex $v$ such that $L(v)$ contains a color \emph{not} used by $c$, say $c_0$.  Define a coloring $c'$ in which $c'(v)=c_0$ and $c'=c$ for all other vertices.  Since color $c_0$ was not used by coloring $c$, we may conclude that $w_{c'}$ is lexicographically less than $w_c$.  This contradicts the assumption that $c$ is lex-min, so it must be that $c$ uses at least 4 colors.
\end{proof}

The remainder of this section is organized as follows:  We start by assuming that $c$ is a lex-min $L$-coloring of $\Pi_n$, $n\geq 6$, which is not $\left\lceil 2n/3\right\rceil$-bounded. In Section~\ref{subsec:bluereddifferby0or1}, we use reducible configurations and a counting argument to prove that the two largest color classes of $c$ must differ in cardinality by at least 2.  Then in Section~\ref{subsec:blue2bigger}, we use additional reducible configurations and a discharging argument to prove that $c$ must be $\left\lceil 2n/3\right\rceil$-bounded, and so $\Pi_n$, $n\geq 6$, is equitably 3-choosable.

\subsection{Two Large Color Classes}
\label{subsec:bluereddifferby0or1}

In this subsection we prove that if $c$ is lex-min and not $\left\lceil 2n/3\right\rceil$-bounded, then its two largest color classes must differ in size by at least 2.  To simplify the notation, let's say that Blue is the most common color and Red is the second most common color.  We assume that $|\text{Blue}|\geq \left\lceil 2n/3\right\rceil +1$ and either $|\text{Red}|=|\text{Blue}|$ or $|\text{Red}|=|\text{Blue}|-1$.  First, we determine how much larger $|\text{Red}|$ and $|\text{Blue}|$ are than the other color classes.

\begin{lemma}\label{lem:blueredatleast2bigger}
    If $c$ is lex-min and not $\left\lceil 2n/3\right\rceil$-bounded, and $|\text{Red}|=|\text{Blue}|$ or $|\text{Red}|=|\text{Blue}|-1$, then $|\text{Blue}|$ and $|\text{Red}|$ are at least 2 larger than every color class.
\end{lemma}
\begin{proof}
    Recall that by Lemma~\ref{lem:cuses4colors}, $c$ has at least 4 nonempty color classes.  Suppose $|\text{Red}|=|\text{Blue}|-1$, and let $C$ denote an arbitrary third color class.  Then because $\Pi_n$ has $2n$ vertices and $|\text{Blue}|\geq \left\lceil 2n/3\right\rceil +1$, we may conclude that
    \begin{align*}
        |C| &\leq 2n - |\text{Blue}|-|\text{Red}| - 1 \\
            &\leq 2n - 2\left\lceil 2n/3\right\rceil -2 \\
            &\leq 2n/3 -2\\
            &\leq \left\lceil 2n/3\right\rceil -2
    \end{align*}
    Note that $\left\lceil 2n/3\right\rceil -2 \leq |\text{Blue}|-3$ and $\left\lceil 2n/3\right\rceil -2 \leq |\text{Red}|-2$.
    So when $|\text{Red}|=|\text{Blue}|-1$, Blue contains at least 3 more vertices and Red contains at least 2 more vertices than every other color class.
    
    A similar bounding argument shows that when $|\text{Red}|=|\text{Blue}|$, Red and Blue both contain at least 4 more vertices than every other color class.
\end{proof}

Next, we identify four color configurations that cannot occur in $c$.  Note that in these diagrams, a vertex that is not colored in could have any color which is not already in use by one of its neighbors.
\begin{figure}[!htb]
\begin{center}
    \begin{minipage}{0.3\textwidth}\begin{center}
    \includegraphics[width=0.75\linewidth]{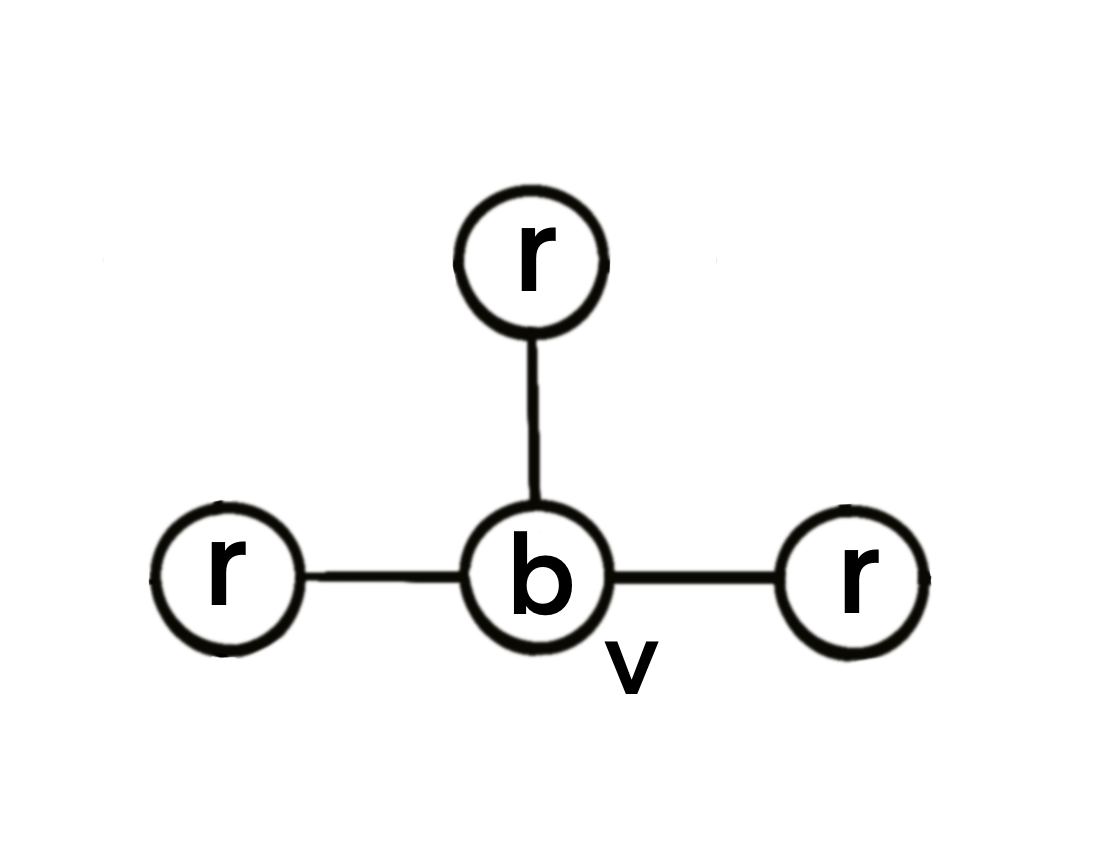}
    \captionof*{figure}{Configuration $F_1$}\end{center}
    \end{minipage}\hspace{0.1\textwidth}
    \begin{minipage}{0.3\textwidth}\begin{center}
    \includegraphics[width=0.75\linewidth]{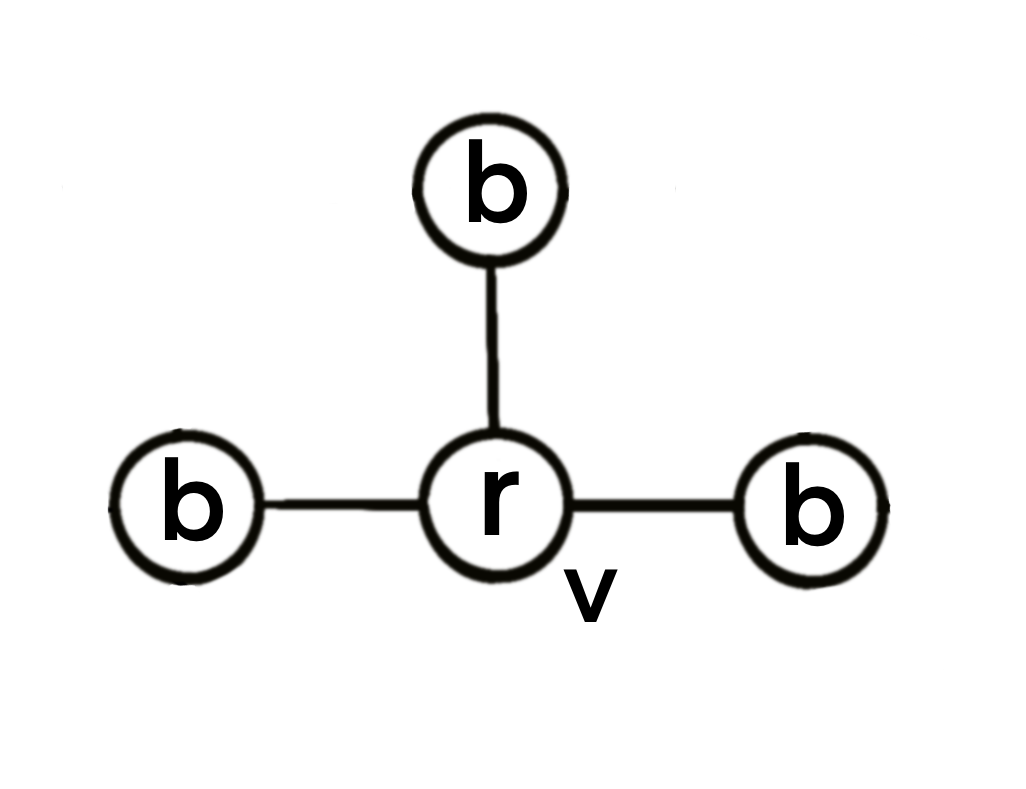} 
    \captionof*{figure}{Configuration $F_2$}\end{center}  
    \end{minipage}
    
    \phantom{line}
    
    \begin{minipage}{0.3\textwidth}
    \includegraphics[width=\linewidth]{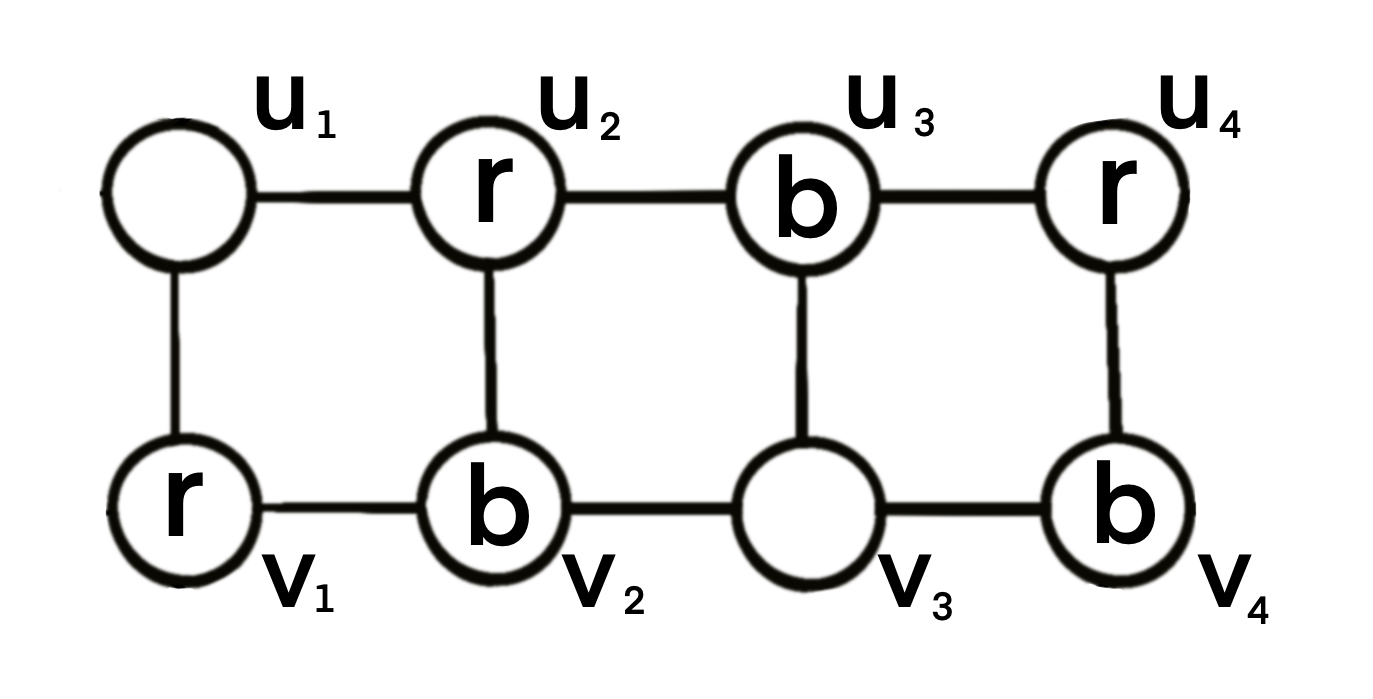} 
    \captionof*{figure}{Configuration $F_3$}   
    \end{minipage}\hspace{0.1\textwidth}
    \begin{minipage}{0.3\textwidth}
    \includegraphics[width=\linewidth]{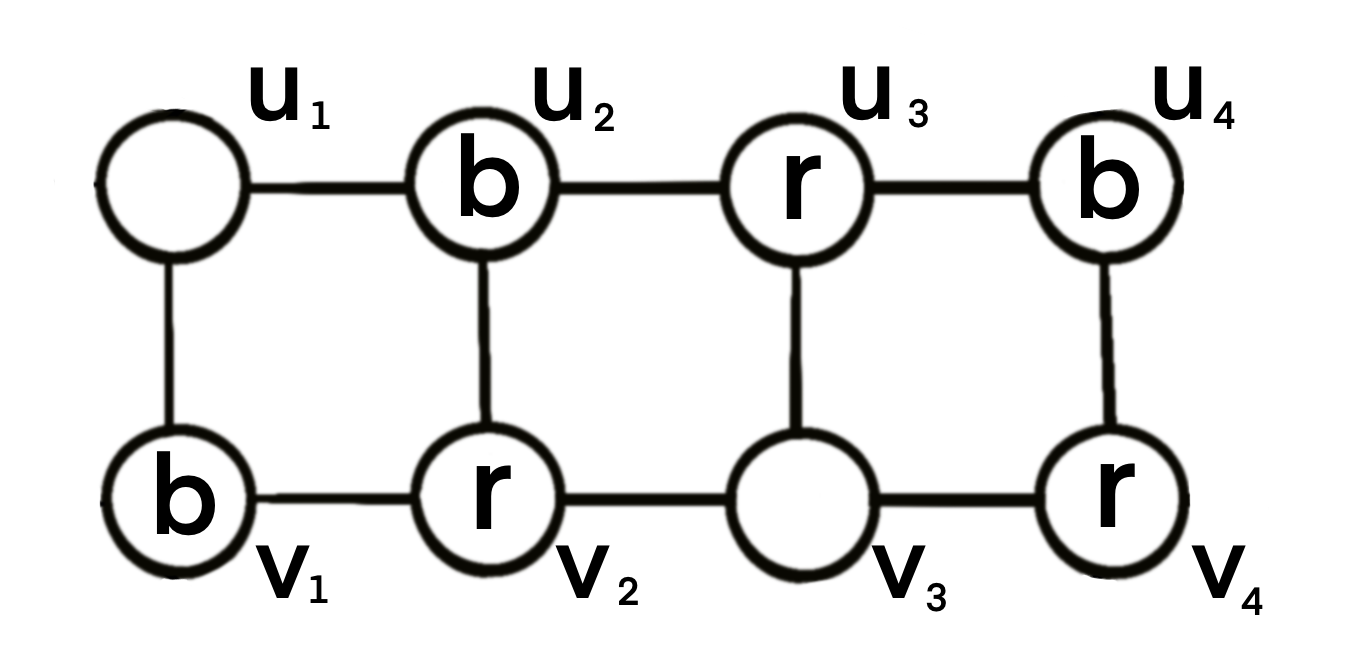} 
    \captionof*{figure}{Configuration $F_4$}   
    \end{minipage}
\caption{Reducible configurations for Lemma~\ref{lem:reducibleconfigsDiff0or1}}
\label{fig:reducibleconfigs1}
\end{center}
\end{figure}

\begin{lemma}\label{lem:reducibleconfigsDiff0or1}
    If $c$ is lex-min and not $\left\lceil 2n/3\right\rceil$-bounded, and $|\text{Red}|=|\text{Blue}|$ or $|\text{Red}|=|\text{Blue}|-1$, then the color configurations in Figure~\ref{fig:reducibleconfigs1} do not occur in $c$.
\end{lemma}
\begin{proof}
    In configurations $F_1$ and $F_2$, notice that $L(v)$ contains a color which is neither $r$ nor $b$.  Define a new coloring $c'$ in which $c'(v)\not\in\{b,r\}$ and $c'=c$ for all other vertices.  Because either $|\text{Red}|$ or $|\text{Blue}|$ decreases by 1 and a third color class increases by 1, $w_{c'}$ is lexicographically less than $w_c$.  This is a contradiction, so configurations $F_1$ and $F_2$ cannot occur in the lex-min coloring $c$.
    
    In configurations $F_3$ and $F_4$, we attempt to recolor one of vertices $u_2$, $u_3$, and $v_2$.  If one of these vertices' lists contains a color other than $r$, $b$, and the blank neighbor's color, then we may recolor that vertex and obtain a lexicographically smaller color word.  This is a contradiction, so it must be that $L(u_2)=\{b,r,c(u_1)\}$, and $L(u_3)=L(v_2)=\{b,r,c(v_3)\}$.  Note that this implies $c(v_3)$ is neither $r$ nor $b$.
    
    In $F_3$, there must be a color $c_0$ in $L(v_3)$ such that $c_0\not\in\{b,c(v_3)\}$.  Define a new coloring $c'$ in which $c'(v_3)=c_0$, $c'(u_3)=c'(v_2)=c(v_3)$, $c'(u_2)=b$, and $c'=c$ for all other vertices.  If $c_0=r$, then $|\text{Red}|$ stays the same, $|\text{Blue}|$ decreases by 1, and the color class corresponding to $c(v_3)$ increases by 1.  If $c_0\ne r$, then $|\text{Red}|$ and $|\text{Blue}|$ both decrease by 1, and the color classes corresponding to $c(v_3)$ and $c'(v_3)$ both increase by 1.  Either way, $w_{c'}$ is lexicographically less than $w_c$. This contradicts the assumption that $c$ was a lex-min coloring, so $F_3$ cannot occur in $c$.
    
    The argument for $F_4$ is similar.  We let $c_0$ denote a color in $L(v_3)$ such that $c_0\not\in\{r,c(v_3)\}$, then we define $c'$ in the same way. To determine that $w_{c'}$ is lexicographically less than $w_c$, we consider cases $c_0=b$ and $c_0\ne b$.  In both cases, we contradict the assumption that $c$ was lex-min, so $F_4$ cannot occur in $c$.
\end{proof}

\begin{lemma}\label{lem:6consecrungs}
    If $c$ is lex-min and not $\left\lceil 2n/3\right\rceil$-bounded, then there does not exist a subgraph of 6 consecutive rungs of $\Pi_n$, $n\geq 6$, in which 9 or more of the vertices are colored red or blue.
\end{lemma}
\begin{proof}
    Suppose there is a subgraph $H$ of 6 consecutive rungs of $\Pi_n$, $n\geq 6$, in which 9 or more of the vertices are colored red or blue.  $H$ cannot contain 6 blue vertices and 3 or more red vertices because then reducible configuration $F_2$ would occur.  Similarly, if $H$ contained 6 red vertices and at least 3 blue vertices, reducible configuration $F_1$ would occur.  So the only possible counts for red and blue vertices are: 5 blue, 4 red; 5 blue, 5 red; and 4 blue, 5 red.
    
    Suppose $H$ contains 5 blue vertices.  Then $H$ looks like one of the diagrams in Figure~\ref{fig:6rungsubgraphs}.  Note that the blank vertices cannot be blue and must be properly colored.
    \begin{figure}[!ht]
    \begin{center}
    \begin{minipage}{0.32\textwidth}
    \includegraphics[width=\linewidth]{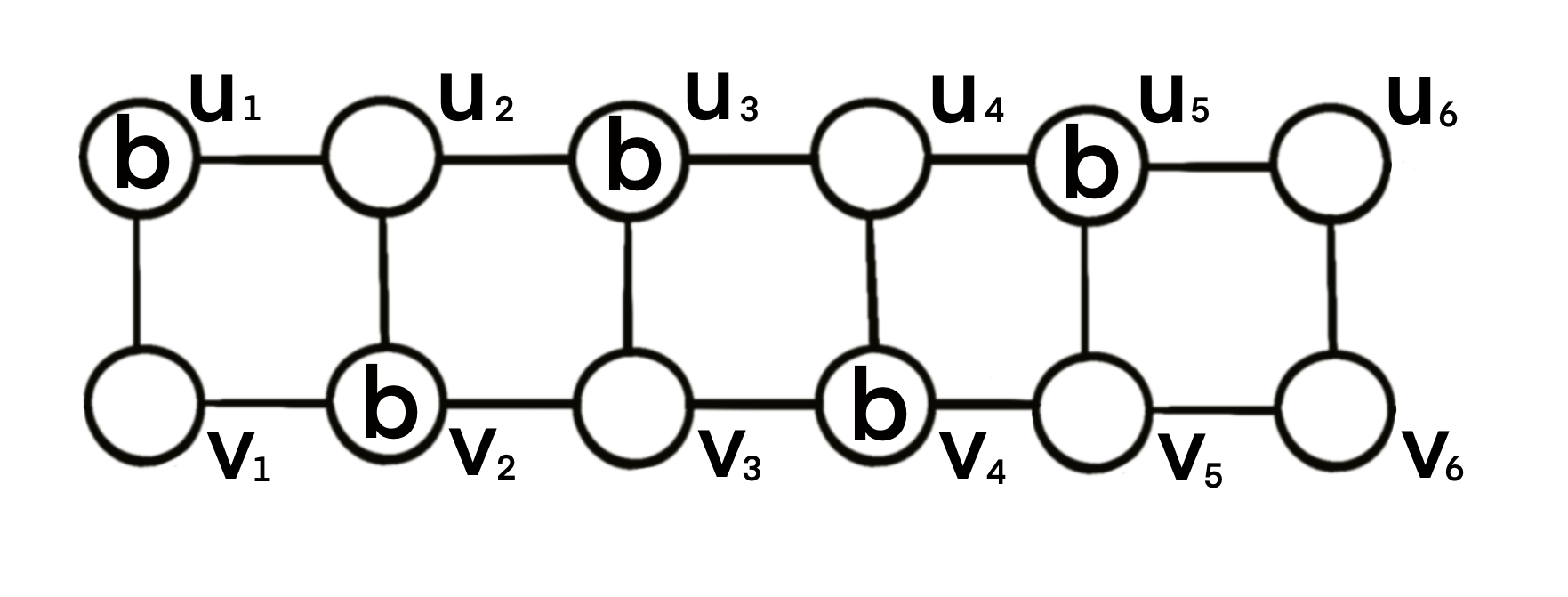}
    \captionof*{figure}{Case 1}
    \end{minipage}\hfill
    \begin{minipage}{0.32\textwidth}
    \includegraphics[width=\linewidth]{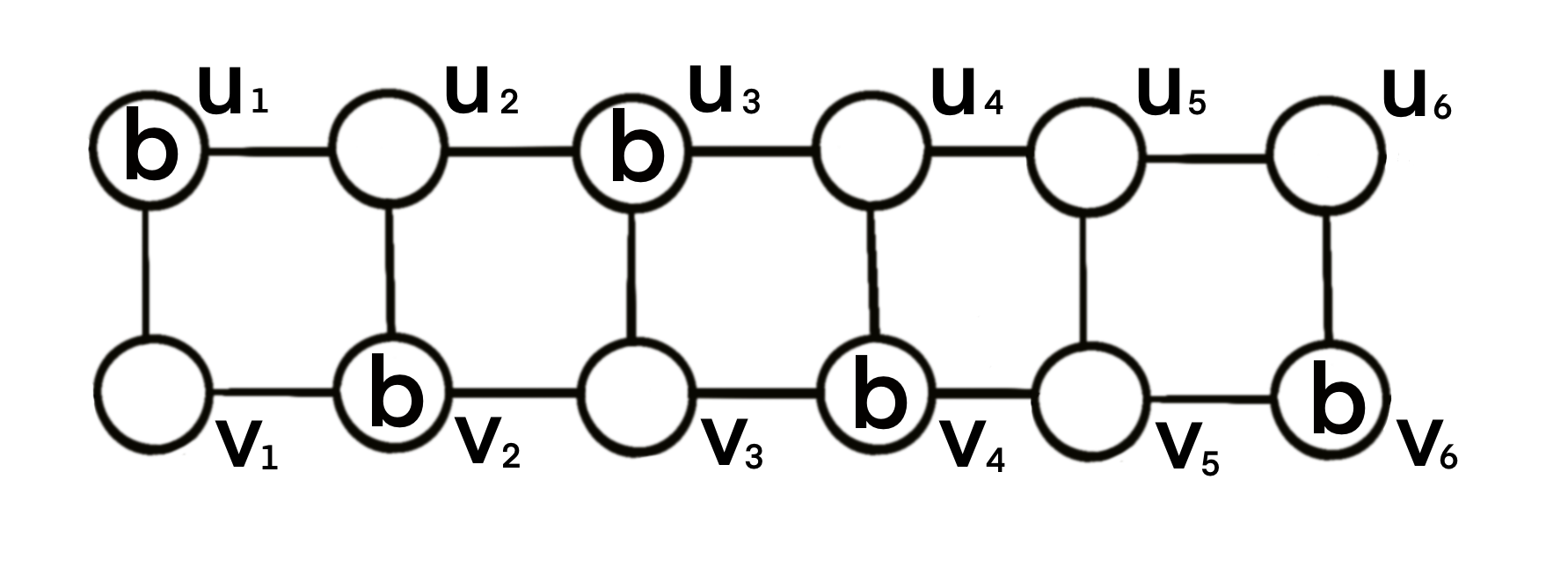} 
    \captionof*{figure}{Case 2}   
    \end{minipage}\hfill
    \begin{minipage}{0.32\textwidth}
    \includegraphics[width=\linewidth]{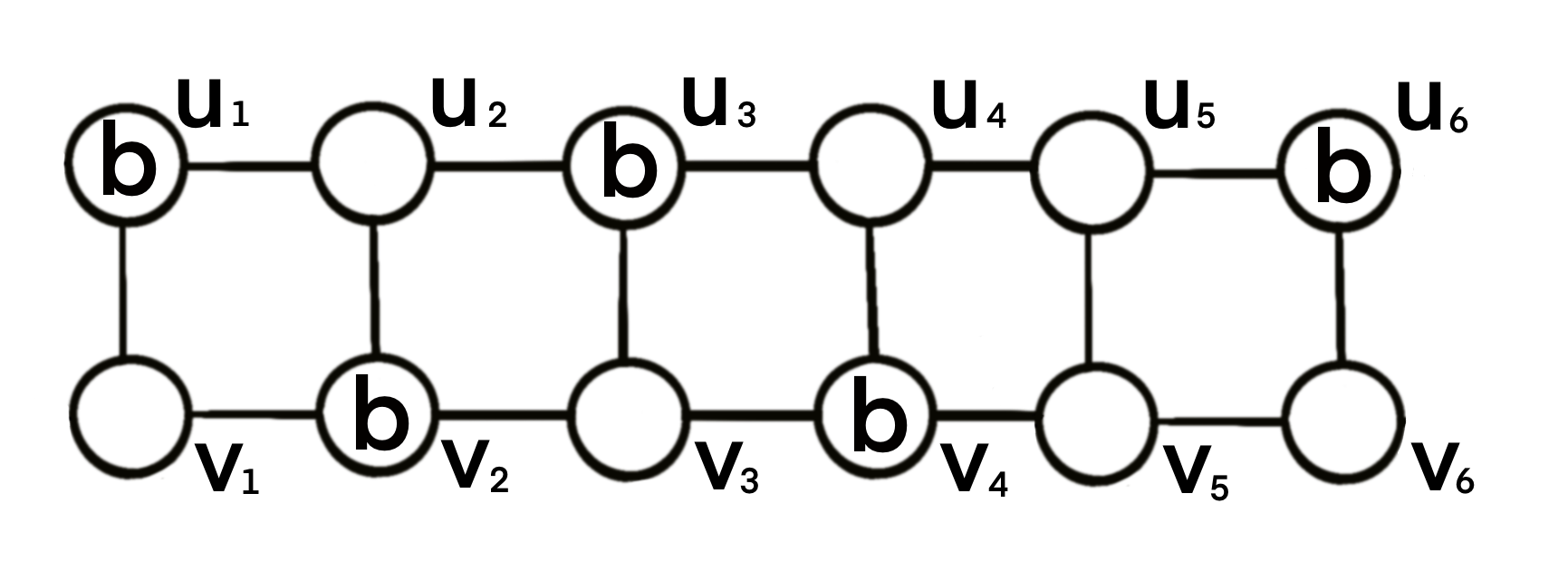} 
    \captionof*{figure}{Case 3}   
    \end{minipage} \\
    \begin{minipage}{0.32\textwidth}
    \includegraphics[width=\linewidth]{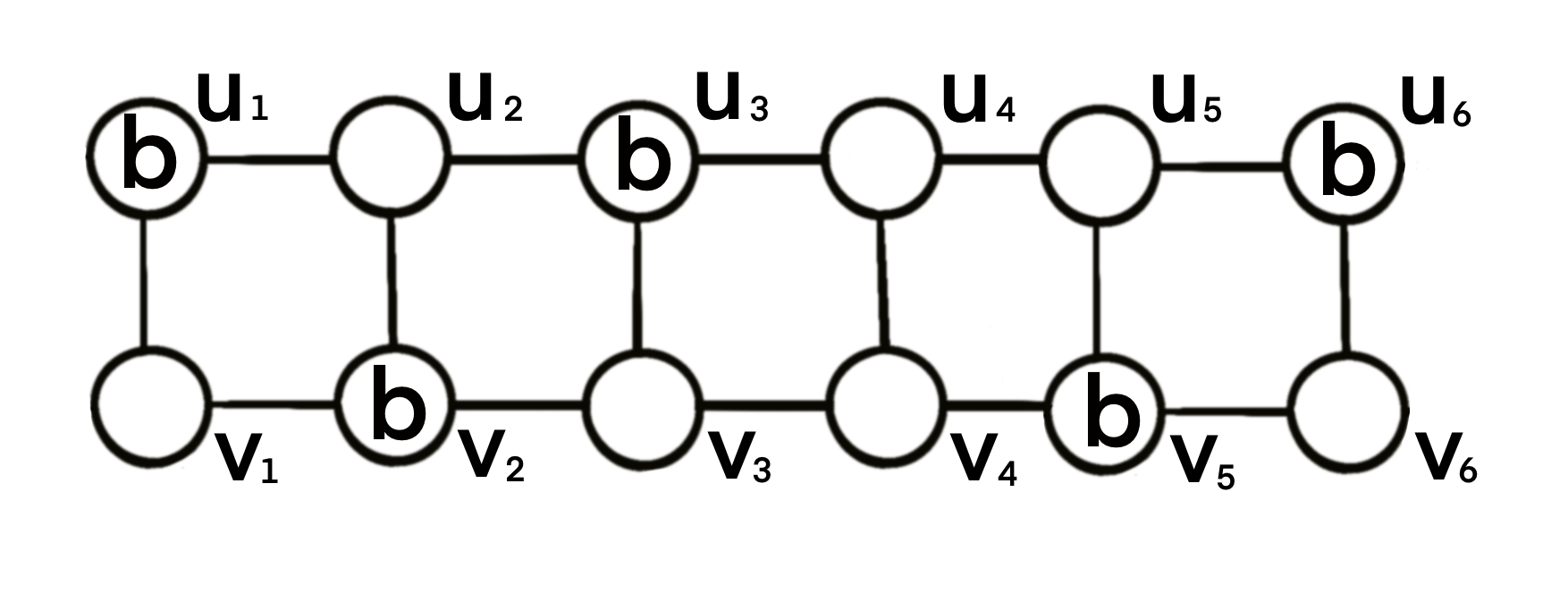} 
    \captionof*{figure}{Case 4}   
    \end{minipage}\hspace{0.1\textwidth}
    \begin{minipage}{0.32\textwidth}
    \includegraphics[width=\linewidth]{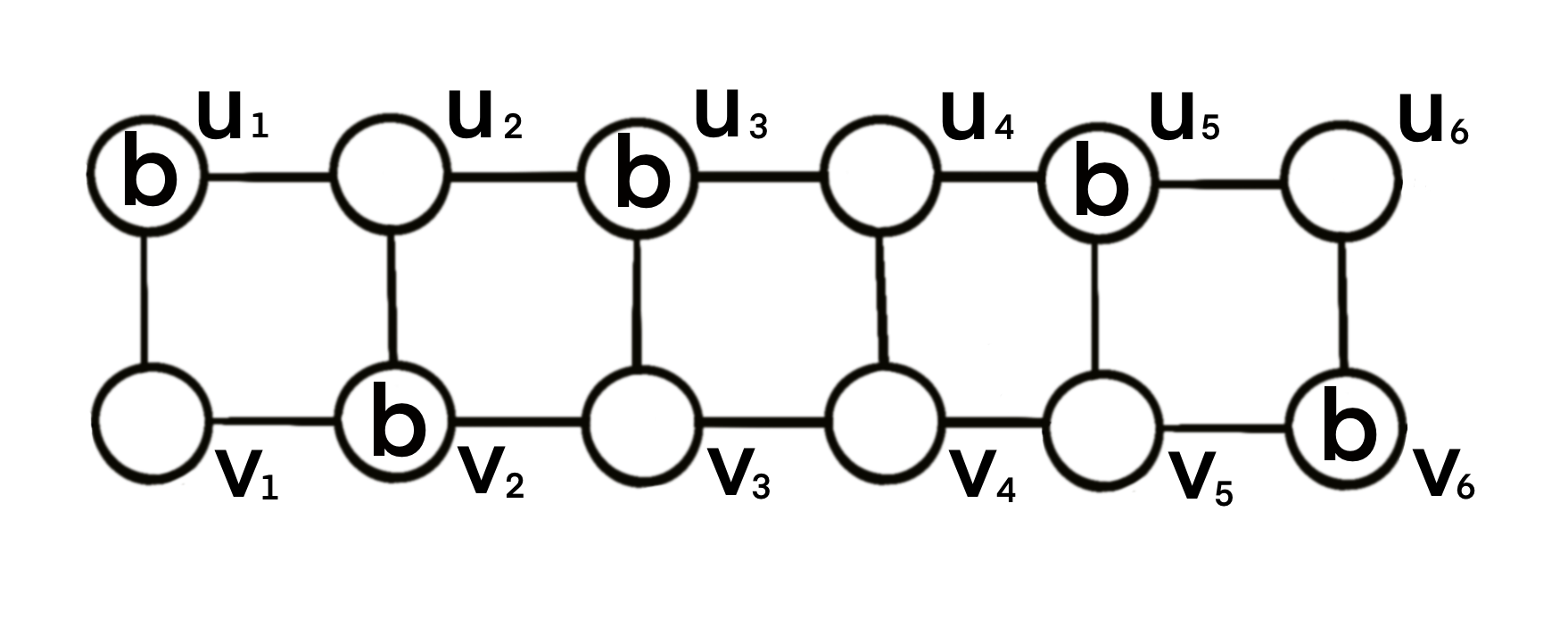} 
    \captionof*{figure}{Case 5}   
    \end{minipage}
\caption{Potential colorings of 6 consecutive rungs.}
\label{fig:6rungsubgraphs}
\end{center}
\end{figure}
    
    In Case 1, the blue vertices are arranged in 5 consecutive rungs.  If 4 of the blank vertices are colored red, and the coloring is proper, then at least one of $u_2$, $u_4$, and $v_3$ would be forced to be red.  However, this creates a copy of the reducible configuration $F_2$, so this arrangement of blue vertices cannot occur.
    
    In Case 2, the blue vertices are arranged in 4 consecutive rungs with the fifth blue vertex being $v_6$.  If $u_2$ or $v_3$ is red, then reducible configuration $F_2$ occurs, which is a contradiction. So if 4 of the blank vertices are colored red, it must be that $v_1$, $u_4$, $v_5$, and $u_6$ are red.  However, this creates a copy of the reducible configuration $F_4$, so this arrangement of blue vertices cannot occur.
    
    In Case 3, the blue vertices are arranged in 4 consecutive rungs with the fifth blue vertex being $u_6$.  If 4 of the blank vertices are colored red, and the coloring is proper, then at least one of $u_2$ and $v_3$ would be forced to be red.  However, this creates a copy of $F_2$, so this arrangement of blue vertices cannot occur.
    
    In Case 4, the blue vertices are arranged in 3 consecutive rungs, with the fourth and fifth blue vertices being $v_5$ and $u_6$. To avoid $F_2$, $u_2$ cannot be red.  The largest independent set of vertices which are neither blue nor $u_2$ is 4, so there cannot be 5 red vertices in this case.  If $v_1$, $v_3$, $u_4$, and $v_6$ are red, a copy of $F_3$ is created, which is a contradiction.  If $v_1$, $v_4$, $u_5$, and $v_6$ are red, a copy of $F_1$ is created, which is a contradiction.  So it must be that $v_1$, $v_3$, $u_5$, and $v_6$ are the only red vertices in this case.  Notice that $u_4$ and $v_4$ are neither red nor blue, and $c(u_4)\ne c(v_4)$.  An updated diagram for Case 4 is shown in Figure~\ref{fig:updatedcase4}.  For simplicity, we will say $c(u_4)=g$ and $c(v_4)=y$.
    \begin{figure}[!ht]
    \centering
    \includegraphics[width=2.5in]{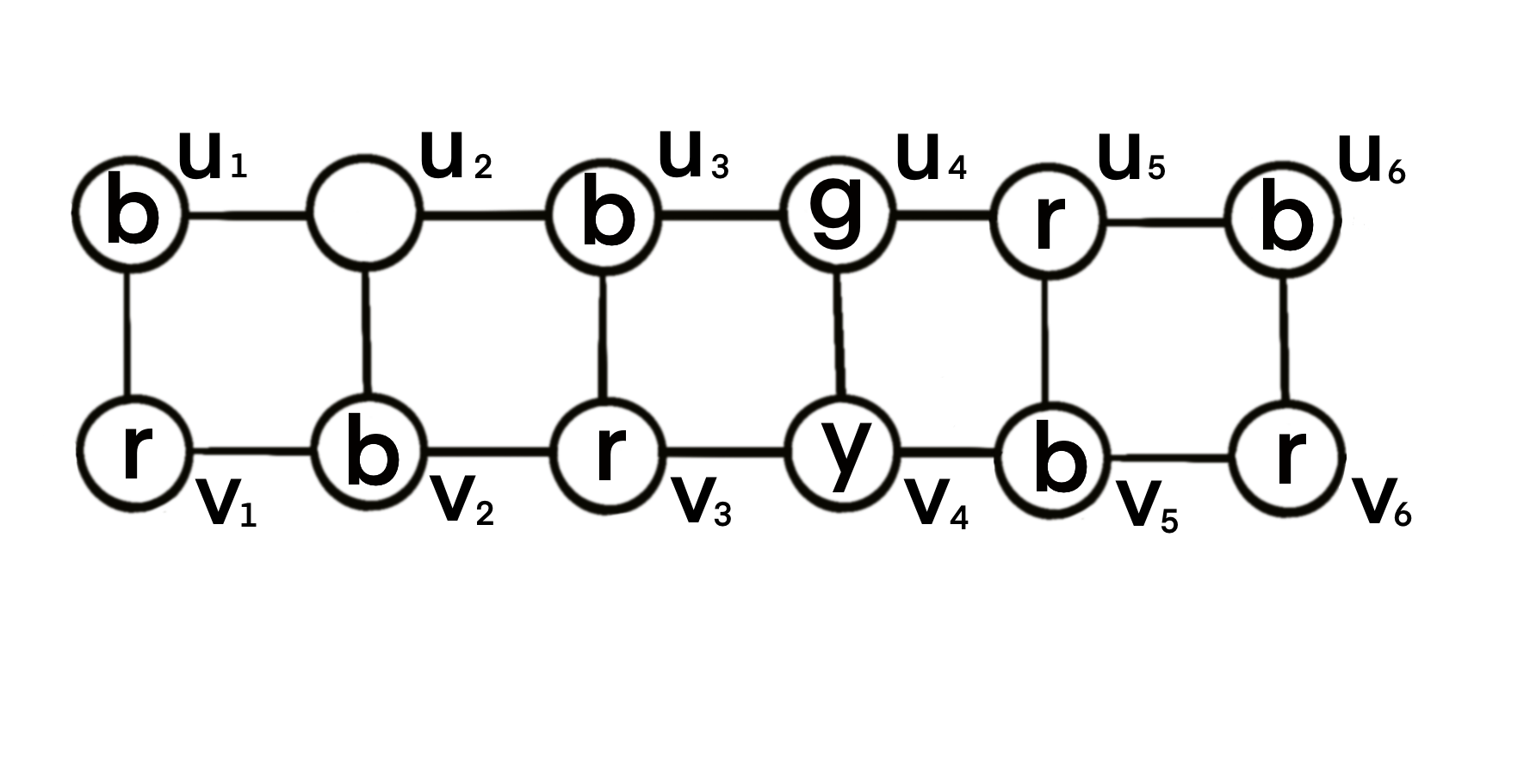}
    \caption{Updated coloring for Case 4}
    \label{fig:updatedcase4}
    \end{figure}
    
    If we were able to change any blue or red vertex to a color other than red and blue, we would contradict the assumption that $c$ is lex-min.  Therefore, $L(u_5)=\{b,g,r\}$ and $L(v_3)=L(v_5)=\{b,r,y\}$. It must be that $L(v_4)\subseteq\{b,g,r,y\}$, because otherwise we could define a new coloring $c'$ in which $c'(v_5)=y$, $c'(v_4)\not\in\{b,g,r,y\}$, and $c'=c$ for all other vertices.  The color word of $c'$ is lexicographically less than $w_c$, which is a contradiction.  For similar reasons, $L(u_4)\subseteq\{b,g,r,y\}$.  
    
    Now, notice that $L(u_4)=\{b,g,y\}$.  If not, then $r\in L(u_4)$, so we could define a new coloring $c'$ in which $c'(u_5)=g$, $c'(u_4)=r$, and $c'=c$ for all other vertices.  The color word of $c'$ is equal to $w_c$, so they are both lex-min colorings.  However, $c'$ contains a copy of $F_3$, which is a contradiction.  Further, notice that $L(v_4)=\{b,r,y\}$.  If not, then $g\in L(v_4)$, and we can define a new coloring $c''$ such that $c''(u_5)=g$, $c''(u_4)=y$, $c''(v_4)=g$, and $c''=c$ for all other vertices.  The color word of $c''$ is less than $w_c$, which is a contradiction. 
    
    Now that we know the color lists for $u_5$, $u_4$, $v_4$, and $v_3$, we can define a new coloring $c'$ in the following way.  Let $c'(u_5)=g$, $c'(u_4)=y$, $c'(v_4)=r$, $c'(v_3)=y$, and $c'=c$ for all other vertices.  The color word of $c'$ is less than $w_c$, which is a contradiction, so the Case 4 color arrangement cannot occur in coloring $c$. 
    
    In Case 5, the blue vertices are arranged in 3 consecutive rungs, with the fourth and fifth blue vertices being $u_5$ and $v_6$.  If $u_2$ is red, then $F_2$ occurs, which is a contradiction.  And if $u_4$, $u_6$, and $v_5$ are all red, then $F_1$ occurs, which is a contradiction.  So there cannot be 5 red vertices in this case.  Further, to avoid $F_1$ and $F_2$ when there are 4 red vertices, it must be that $v_1$, $v_3$, and exactly 2 of $u_4$, $u_6$, and $v_5$ are red.  If $u_4$ is red, a copy of $F_3$ is created, which is a contradiction.  So it must be that $v_1$, $v_3$, $v_5$, and $u_6$ are the only red vertices in this case.  Notice that $u_4$ and $v_4$ are neither red nor blue, and $c(u_4)\ne c(v_4)$.  An updated diagram for Case 5 is shown in Figure~\ref{fig:updatedcase5}.  For simplicity, we will say $c(u_4)=g$ and $c(v_4)=y$. 
    \begin{figure}[!ht]
    \centering
    \includegraphics[width=2.5in]{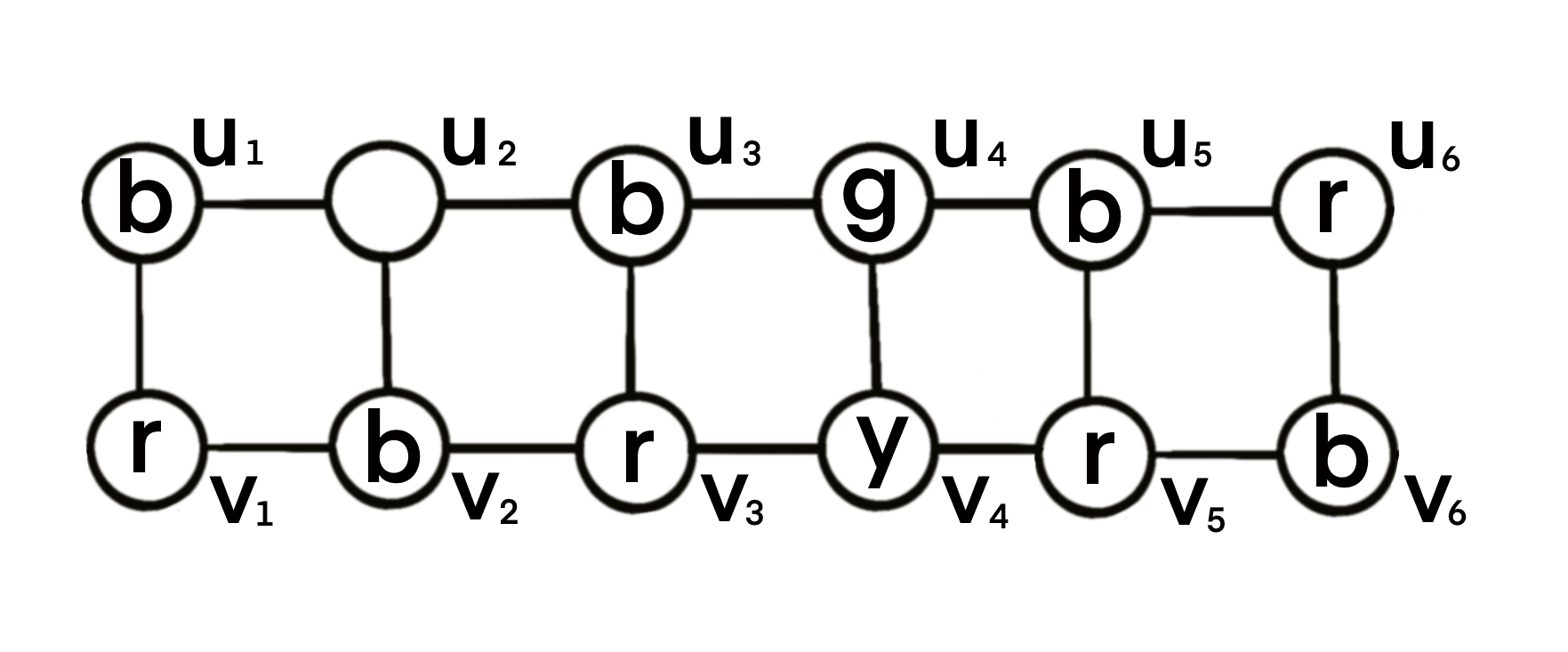}
    \caption{Updated coloring for Case 5}
    \label{fig:updatedcase5}
    \end{figure}
    
    By an argument similar to the one from Case 4, we can determine that $L(u_5) = \{b,g,r\}$, $L(v_3)=L(v_5)=\{b,r,y\}$, $L(v_4)\subseteq\{b,g,r,y\}$, and $L(u_4)\subseteq\{b,g,r,y\}$.  Now, notice that $L(v_4)=\{b,g,y\}$.  If not, then $r\in L(v_4)$, so we could define a new coloring $c'$ in which $c'(v_4)=r$, $c'(v_3)=c'(v_5)=y$, and $c'=c$ for all other vertices.  The color word of $c'$ is lexicographically less than $w_c$, which is a contradiction.  
    
    We know there must exist some $c_o\in L(u_4)$ such that $c_0\not\in\{b,g\}$.  In particular, $c_0\in\{r,y\}$.  So define a new coloring $c'$ in the following way: $c'(u_5)=g$; $c'(u_4)=c_0$; $c'(v_4)=b$;  $c'(v_5)=y$; if $c_0=r$, then $c'(v_3)=y$, otherwise $c'(v_3)=r$; and for all other vertices, $c'=c$.  When switching from $c$ to $c'$, the blue and green color classes stay the same size, the red color class decreases by 1, and the yellow color class increases by 1, so $w_{c'}$ is lexicographically less than $w_c$.  This is a contradiction, so the Case 5 color arrangement cannot occur in coloring $c$.
    
    Similar arguments hold when $H$ contains 5 red vertices and 4 blue vertices.  Simply swap the colors red and blue, swap the reducible configurations $F_1$ and $F_2$, and swap the configurations $F_3$ and $F_4$.  Therefore, when $c$ is lex-min, every subgraph of 6 consecutive rungs of $\Pi_n$ can contain at most 8 vertices that are blue or red.
\end{proof}

\begin{lemma}\label{lem:bluereddiffatleast2}
    If $c$ is lex-min and not $\left\lceil 2n/3\right\rceil$-bounded, then $|\text{Blue}|\geq |\text{Red}|+2$.
\end{lemma}
\begin{proof}
    Suppose $|\text{Blue}|<|\text{Red}|+2$.  That is, suppose $|\text{Red}|=|\text{Blue}|$ or $|\text{Red}| = |\text{Blue}|-1$.
    
    By Lemma~\ref{lem:6consecrungs}, we know that every 6 consecutive rungs must contain at most 8 vertices that are blue or red.  There are $n$ distinct subgraphs of 6 consecutive rungs, and each red or blue vertex is in exactly 6 of these subgraphs.  So we can now bound the number of red and blue vertices in $\Pi_n$, $n\geq 6$, in the following way:
    \begin{align*}
        |\text{Blue}|+|\text{Red}| &\leq \frac{8n}{6} \\
                                   &\leq 2\left\lceil 2n/3\right\rceil \\
                                   &< \left\lceil 2n/3\right\rceil + 1 + \left\lceil 2n/3\right\rceil \\
                                   &\leq |\text{Blue}| + |\text{Red}|.
    \end{align*}
    This is a contradiction, so if $|\text{Blue}| \geq \left\lceil 2n/3\right\rceil +1$, we must have $|\text{Red}|\leq \left\lceil 2n/3\right\rceil -1 \leq |\text{Blue}|-2$.
\end{proof}

\subsection{One Large Color Class}
\label{subsec:blue2bigger}

In Section~\ref{subsec:bluereddifferby0or1}, we determined that the largest color class of a lex-min, non-$\left\lceil 2n/3\right\rceil$-bounded coloring of $\Pi_n$, $n\geq 6$, must contain at least 2 more vertices than every other color class.  In this subsection, we use this fact to finish proving our main result.  To simplify notation, we continue to say Blue is the largest color class.

To begin, we show that the color configurations in Figure~\ref{fig:reducibleconfigs2} do not occur in a lex-min, non-$\left\lceil 2n/3\right\rceil$-bounded coloring.  In these diagrams, the blue vertices must come from the largest color class, but the yellow vertices do not need to come from the second largest color class.  A blank vertex could be any color which is not already in use by its neighbors.

\begin{figure}[!htb]
\begin{center}
    \begin{minipage}{0.28\textwidth}
    \includegraphics[width=1.75in]{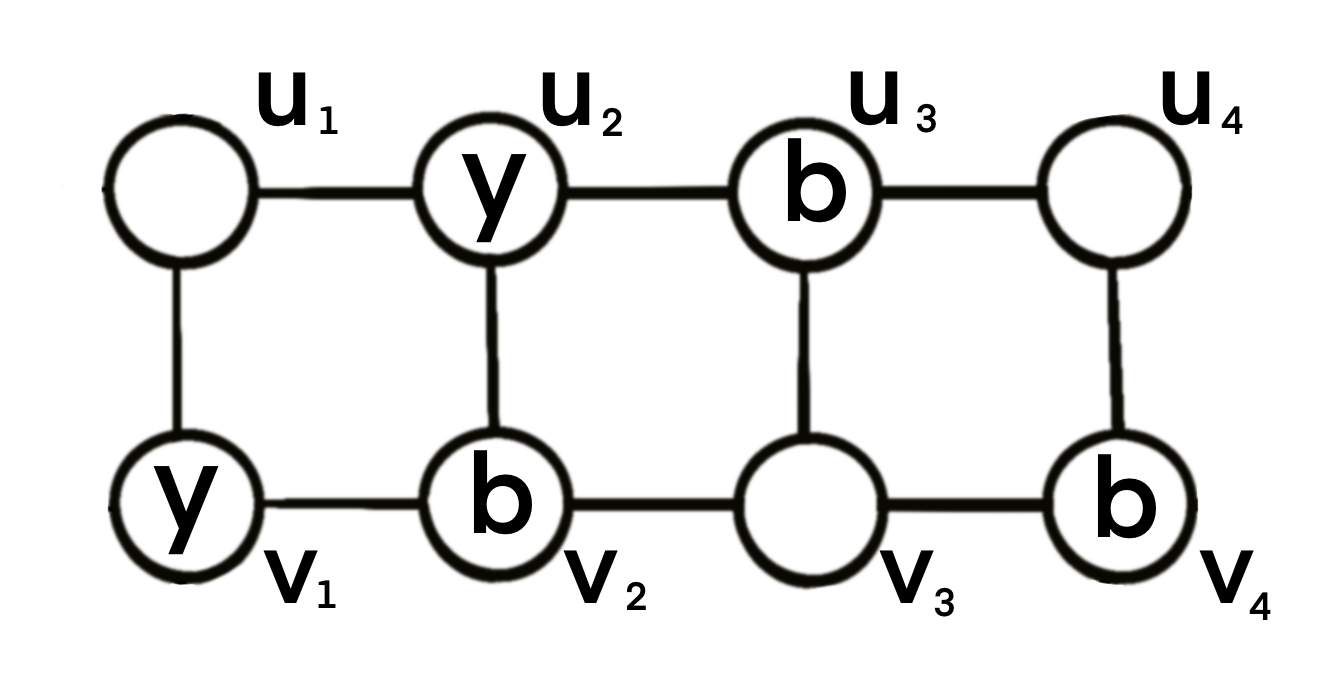} 
    \captionof*{figure}{Configuration $F_5$}   
    \end{minipage}\hspace{0.1\textwidth}
    \begin{minipage}{0.28\textwidth}
    \includegraphics[width=1.75in]{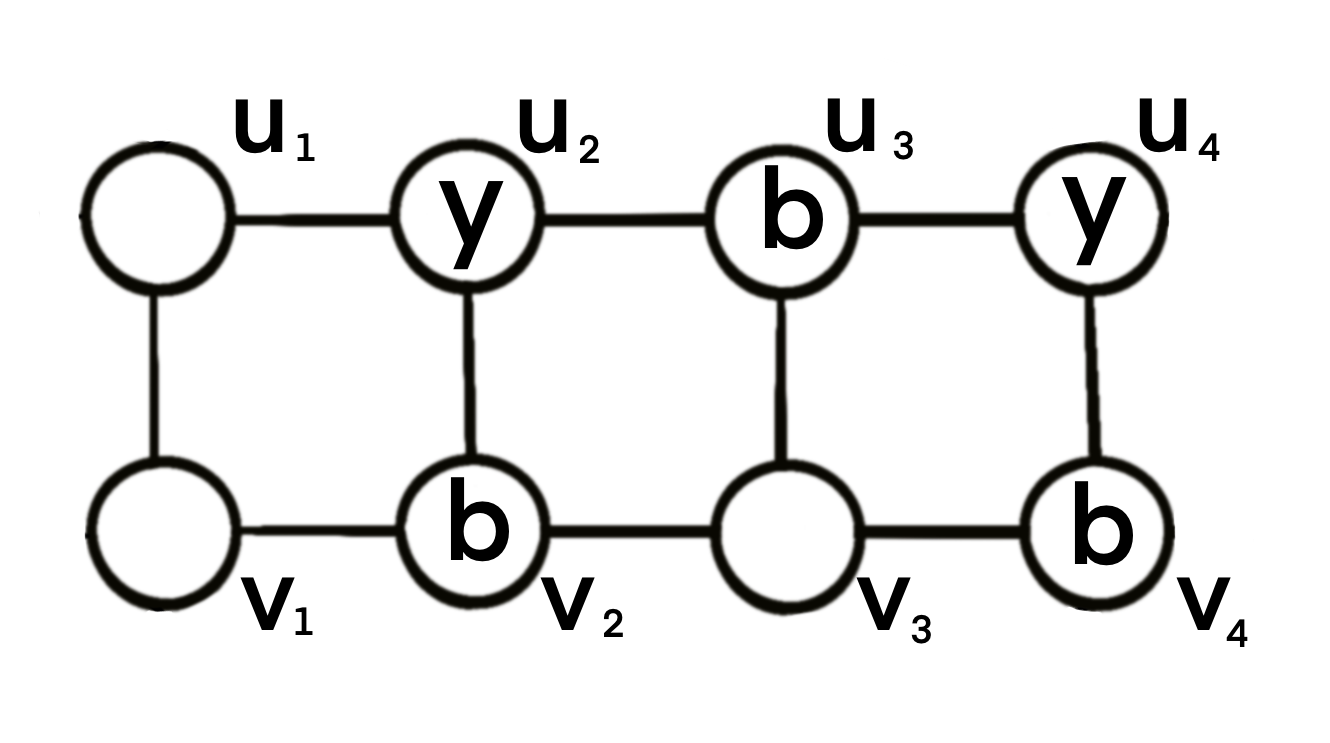} 
    \captionof*{figure}{Configuration $F_6$}   
    \end{minipage}
\caption{Reducible configurations for Lemma~\ref{lem:reducibleconfigsdifferatleast2}}
\label{fig:reducibleconfigs2}
\end{center}
\end{figure}

\begin{lemma}\label{lem:reducibleconfigsdifferatleast2}
    If $c$ is lex-min and not $\left\lceil 2n/3\right\rceil$-bounded, then the color configurations in Figure~\ref{fig:reducibleconfigs2} do not occur.
\end{lemma}
\begin{proof}
    In configuration $F_5$, we attempt to recolor $v_2$.  If $L(v_2)$ contains a color other than $b$, $y$, and $c(v_3)$, then we may recolor it and obtain a lexicographically smaller color word.  This is a contradiction, so it must be that $L(v_2)=\{b,y,c(v_3)\}$.  Note that this implies $c(v_3)$ is neither $b$ nor $y$.  Define a new coloring $c'$ in which $c'(v_2)=c(v_3)$ and $c'(v_3)\not\in\{b,c(v_3)\}$.  Since $|\text{Blue}|$ decreases by 1, $w_{c'}$ is lexicographically less than $w_c$, which is a contradiction.  So $F_5$ cannot occur in the lex-min coloring $c$.
    
    In configuration $F_6$, we attempt to recolor $u_3$. By an argument similar to the $F_5$ case, $L(u_3)=\{b,y,c(v_3)\}$ and $c(v_3)\not\in\{b,y\}$.  Define a new coloring $c'$ in which $c'(u_3)=c(v_3)$ and $c'(v_3)\not\in\{b,c(v_3)\}$.  Then $w_{c'}$ is lexicographically less than $w_c$, which is a contradiction.  So $F_6$ cannot occur in the lex-min coloring $c$.
\end{proof}

Next, we will show that there cannot be 4 consecutive rungs which contain blue vertices.

\begin{lemma}\label{lem:fourconsecutiverungs}
    If $c$ is lex-min and not $\left\lceil 2n/3\right\rceil$-bounded, then there does not exist a subgraph of 4 consecutive rungs of $\Pi_n$, $n\geq 6$, in which 4 vertices are colored blue.
\end{lemma}
\begin{proof}
    Suppose $\Pi_n$ contains a subgraph $H$ of four consecutive rungs with blue vertices.  A diagram of this subgraph is shown in Figure~\ref{fig:4bluerungs}.  Note that the blank vertices cannot be blue because the graph must be properly colored.
    \begin{figure}[!ht]
    \centering
    \includegraphics[width=2in]{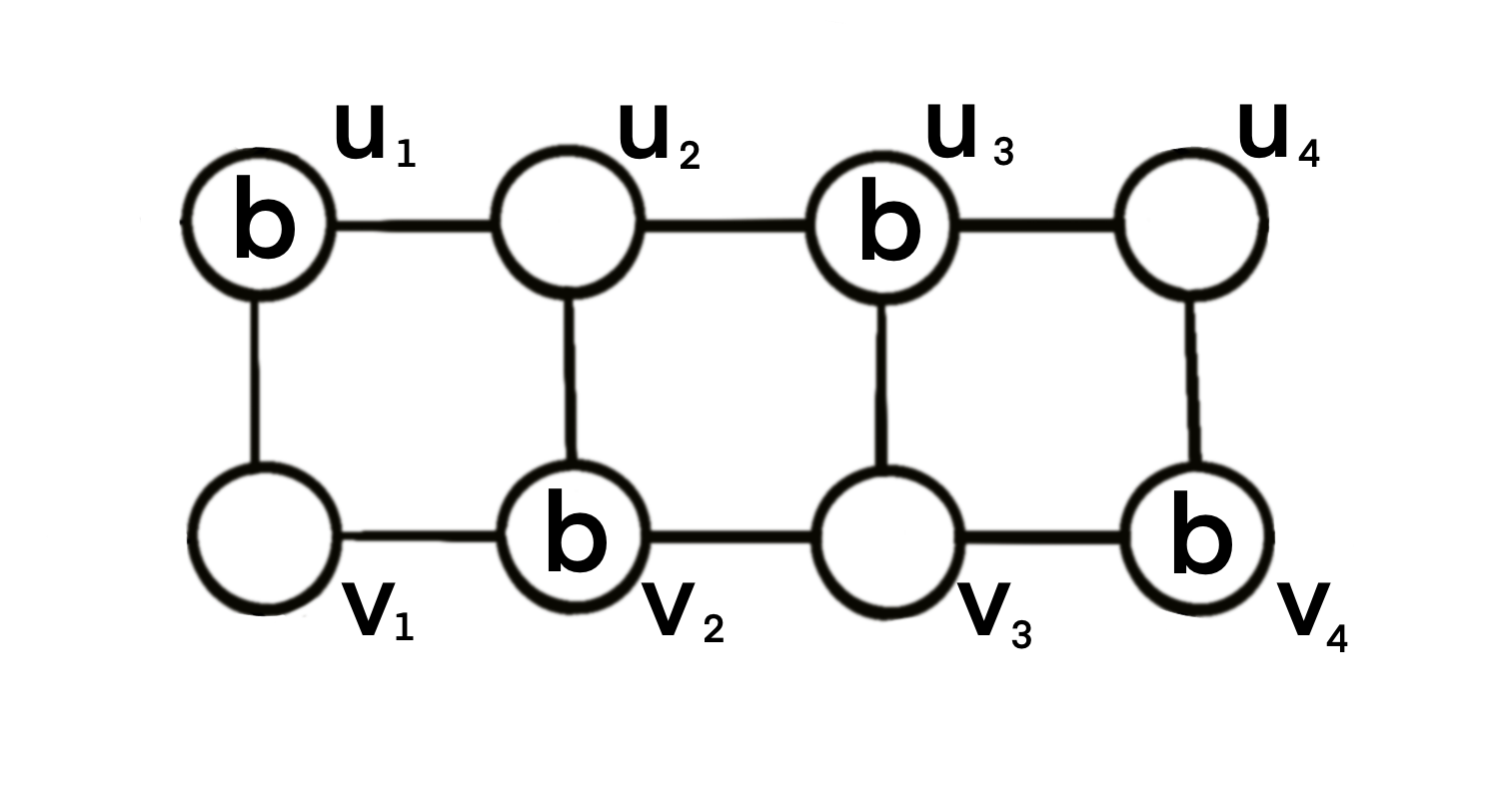}
    \caption{Diagram of subgraph $H$}
    \label{fig:4bluerungs}
    \end{figure}    
    
    To avoid configuration $F_5$, $c(v_1)\ne c(u_2)$ and $c(u_4)\ne c(v_3)$.  To avoid $F_6$, $c(u_2)\ne c(u_4)$ and $c(v_1)\ne c(v_3)$.  Therefore, without a loss of generality, there are four possible ways to color $v_1$, $u_2$, $v_3$, and $v_4$, which can be seen in Figure~\ref{fig:4rungcolorings}.
    \begin{figure}[!htb]
    \begin{center}
    \begin{minipage}{0.3\textwidth}
    \includegraphics[width=\linewidth]{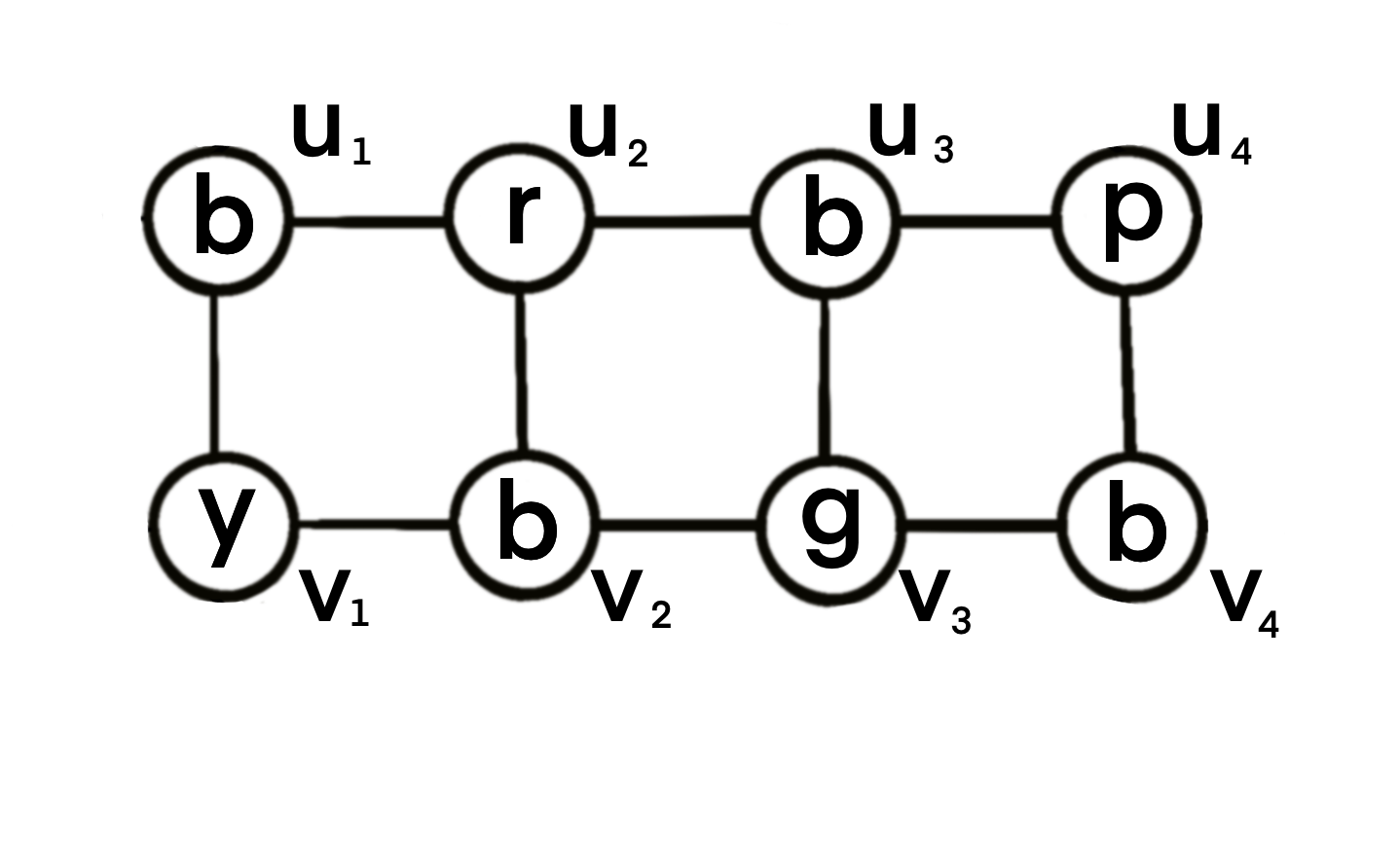}
    \captionof*{figure}{Case 1}
    \end{minipage}\hspace{0.1\textwidth}
    \begin{minipage}{0.3\textwidth}
    \includegraphics[width=\linewidth]{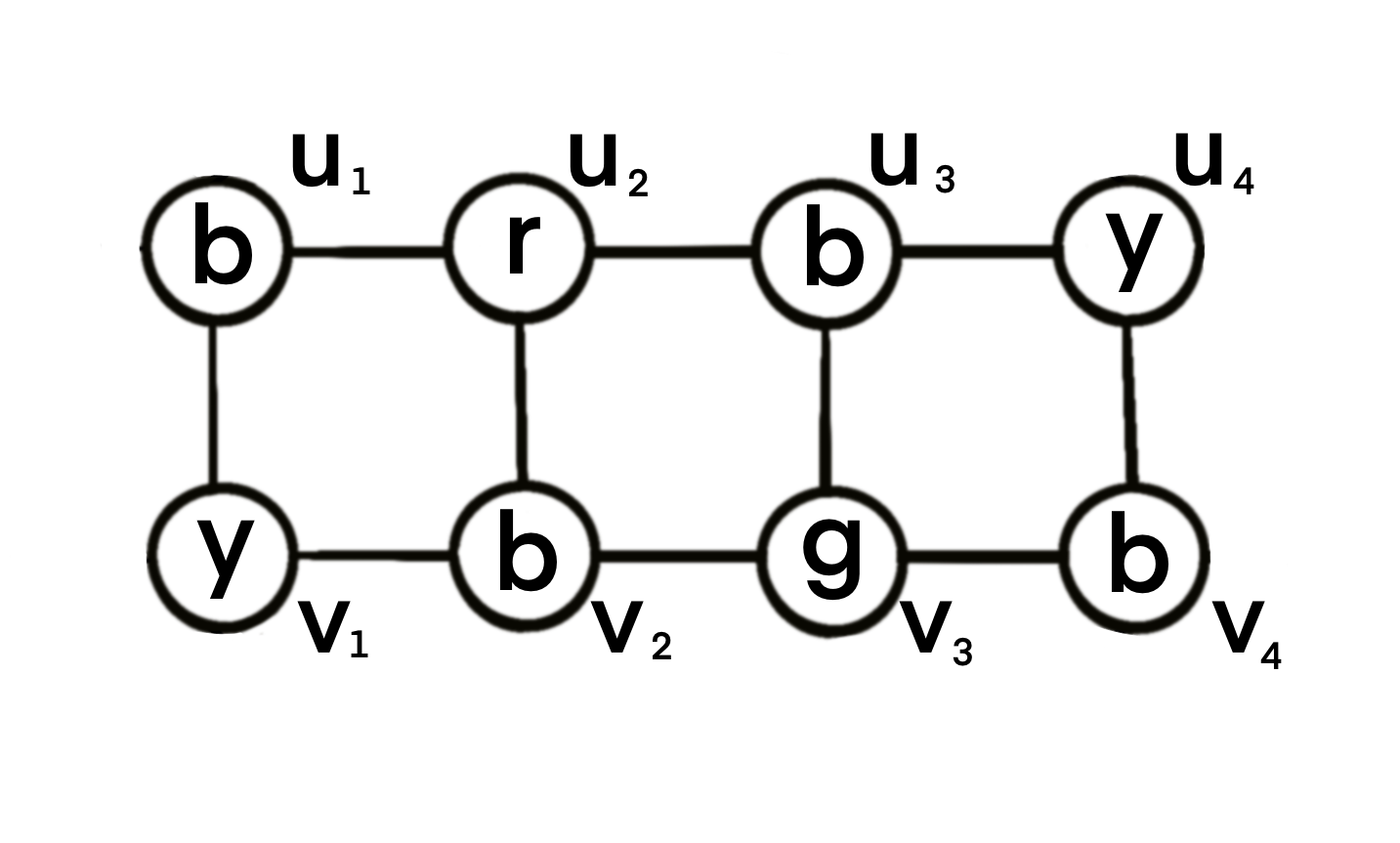} 
    \captionof*{figure}{Case 2}   
    \end{minipage}
    \begin{minipage}{0.3\textwidth}
    \includegraphics[width=\linewidth]{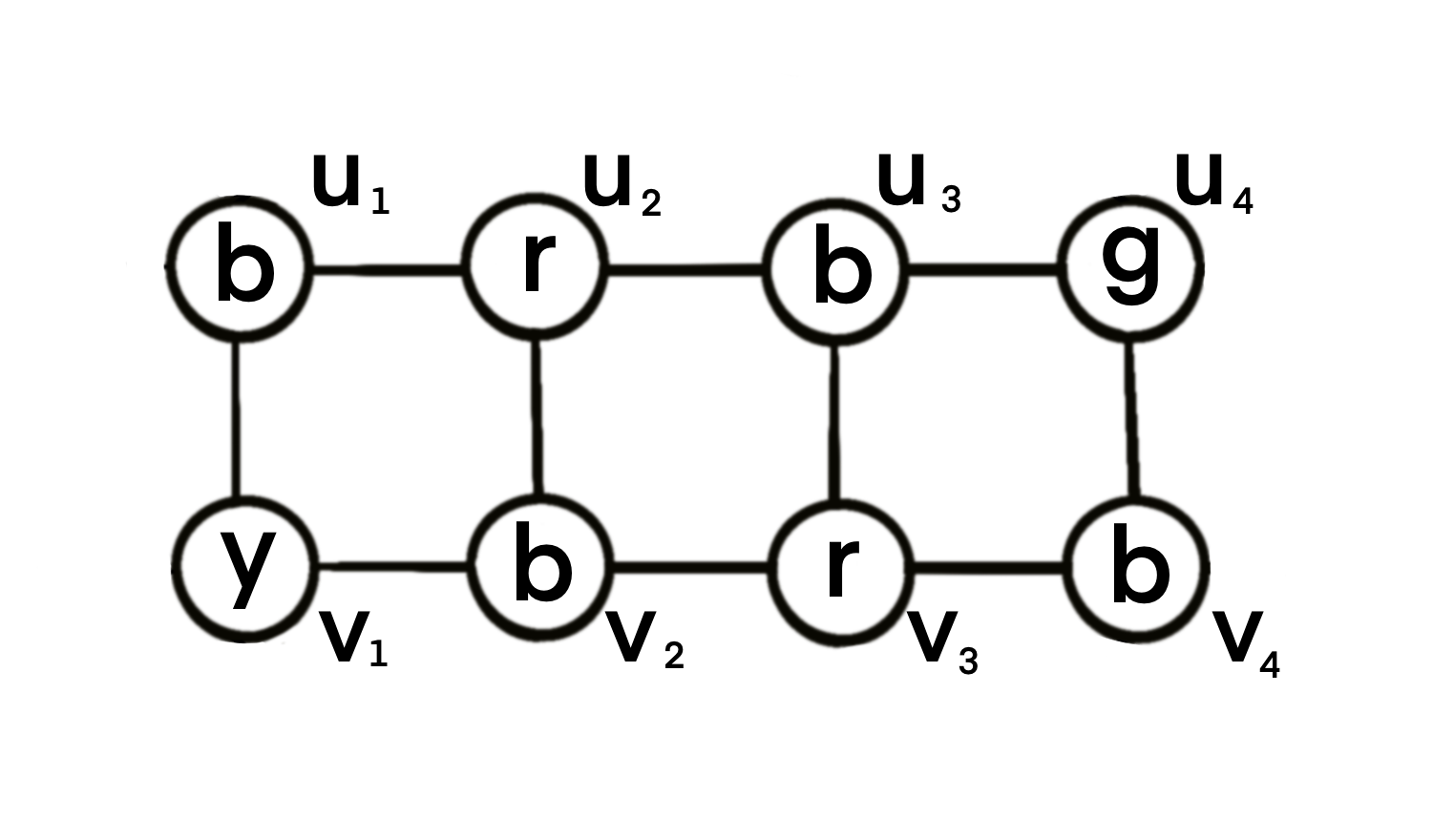} 
    \captionof*{figure}{Case 3}   
    \end{minipage}\hspace{0.1\textwidth}
    \begin{minipage}{0.3\textwidth}
    \includegraphics[width=\linewidth]{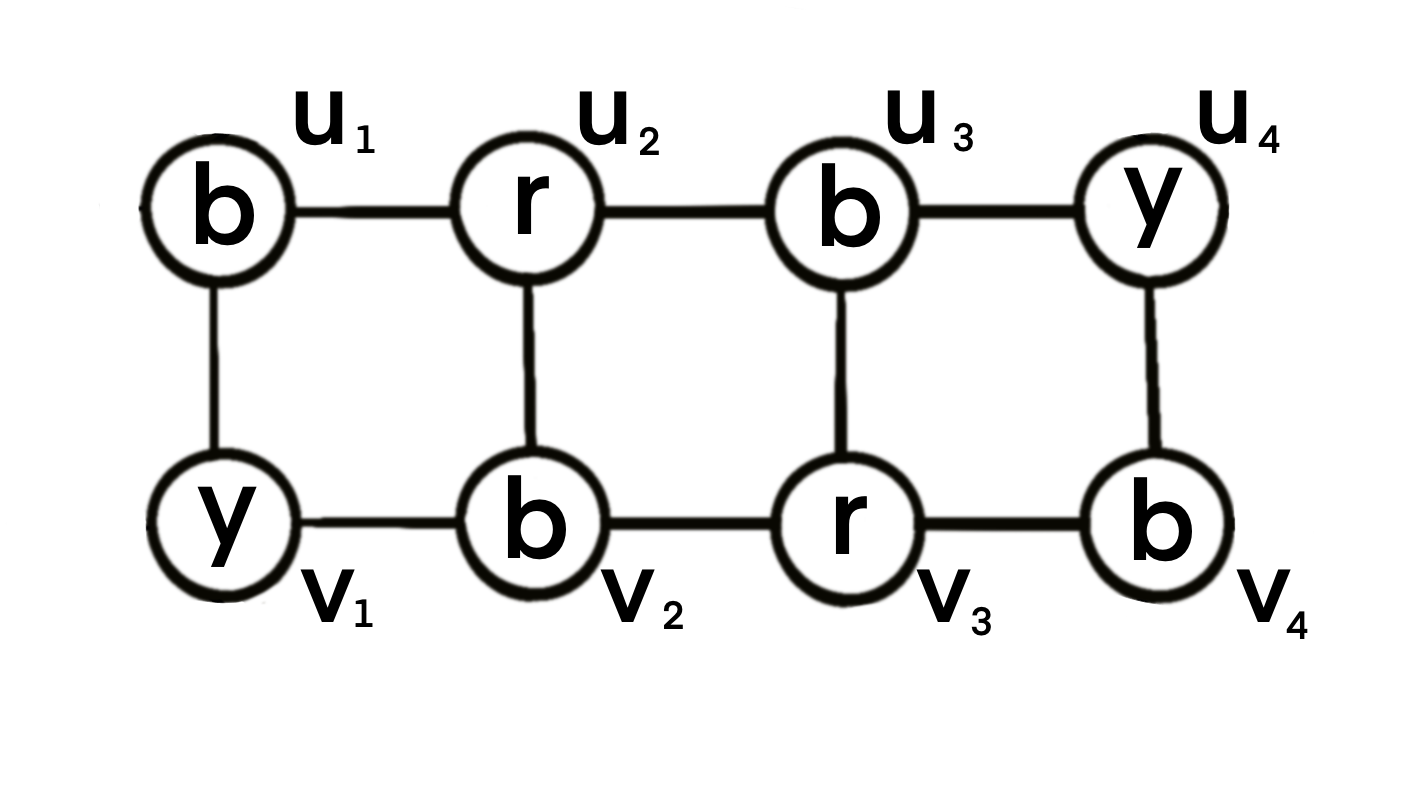} 
    \captionof*{figure}{Case 4}   
    \end{minipage}
    \caption{Possible colorings of $H$}
    \label{fig:4rungcolorings}
    \end{center}
    \end{figure}
    
    In all cases, the color list of every blue vertex must only contain blue and colors utilized by neighbors.  If not, we could recolor one blue vertex to obtain a new coloring whose color word is lexicographically less than $w_c$, which is a contradiction.
    
    In Cases 1 and 2, we know $L(v_2)\subseteq \{b,g,r,y\}$.  If $r\in L(v_2)$, define a new coloring $c'$ in which $c'(v_2)=r$, $c'(u_2)\not\in\{b,r\}$, and $c'=c$ for all other vertices.  If $g\in L(v_2)$, define a new coloring $c''$ in which $c''(v_2)=g$, $c''(v_3)\not\in\{b,g\}$, and $c''=c$ for all other vertices.  Both $w_{c'}$ and $w_{c''}$ are lexicographically less than $w_c$, so we obtain a contradiction and Cases 1 and 2 cannot occur.
    
    In Cases 3 and 4, we know $L(v_2)=\{b,r,y\}$ and $L(u_3)=\{b,r,c(u_4)\}$.  Define a new coloring $c'$ in which $c'(v_2)=c'(u_3)=r$, $c'(u_2),c'(v_3)\not\in\{b,r\}$, and $c'=c$ for all other vertices.  If $c'(u_2)\ne c'(v_3)$, then the blue color class has decreased by 2 and two other color classes each increase by 1.  Since Blue is at least 2 larger than every other color classes, we know that $w_{c'}$ is lexicographically less than $w_c$, which is a contradiction.  So it must be that $c'(u_2)=c'(v_3)$.  If we are forced to make $c'(u_2)=c'(v_3)$, let's say that the new color of $u_2$ and $v_3$ in $c'$ is pink. If pink occurred at least 3 times fewer than blue in $c$, then $w_{c'}$ is still lexicographically less than $w_c$.  So $c'(u_2)=c'(v_3)$, pink was the second most common color class of $c$, pink occurred exactly 2 times fewer than Blue in $c$, and we cannot avoid utilizing pink when defining $c'$.  This causes us to obtain $w_{c'}=w_c$, so we must define $c'$ in a different way to obtain a contradiction.
    
    Since $|\text{Pink}| = |\text{Blue}|-2$ in $c$ and $c$ uses at least 4 colors, it must be that $|\text{Blue}|=\left\lceil 2n/3\right\rceil +1$ and $|\text{Pink}|=\left\lceil 2n/3\right\rceil -1$ in $c$. Further, since we could not avoid coloring $u_2$ and $v_3$ pink, it must be that $L(u_2)=L(v_3)=\{b,p,r\}$.  Recall that $L(v_2)=\{b,r,y\}$ and $L(u_3)=\{b,r,c(u_4)\}$.  We will consider the rung to the left of $u_1$ and $v_1$ as pictured in Figure~\ref{fig:fifthrung}.
    \begin{figure}[!ht]
    \centering
    \includegraphics[width=2in]{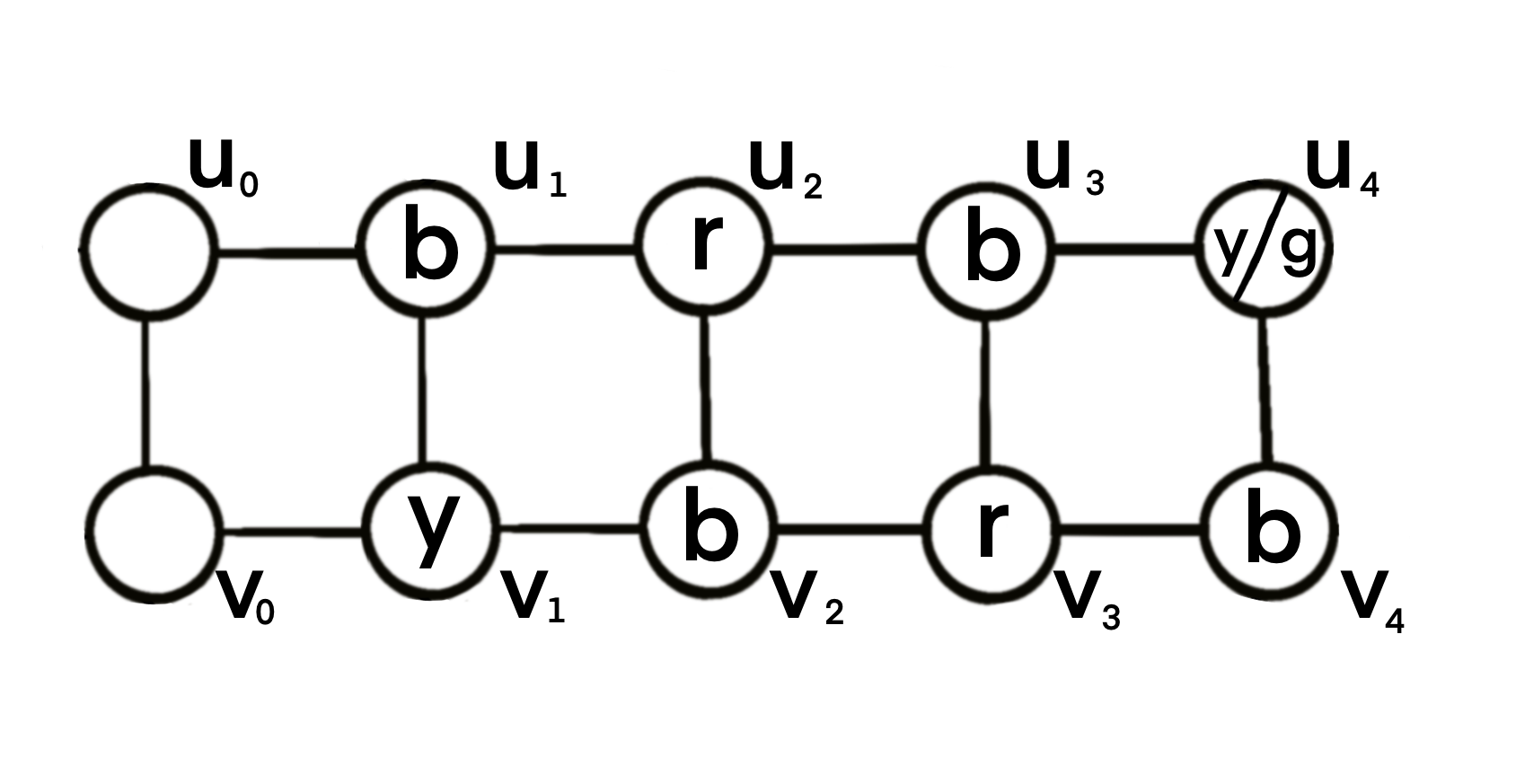}
    \caption{Diagram for Cases 3 and 4}
    \label{fig:fifthrung}
    \end{figure}
    
    Notice that $c(u_0)\ne b$ because $c$ is a proper coloring.  Also, $c(v_0)\ne b$, other wise we could define a new coloring $c'$ in which $c'(v_2)=y$, $c'(v_1)\not\in\{b,y\}$, and $c'=c$ for all other vertices.  Since $w_{c'}$ is less than $w_c$, this gives us a contradiction.  Additionally, notice that $L(v_1)=\{b,y,c(v_0)\}$.  If not, we could obtain a contradiction by defining $c'$ so $c'(v_2)=y$, $c'(v_1)\not\in\{b,y,c(v_0)\}$, and $c'=c$ for all other vertices.  And $L(u_1)\subseteq\{b,r,y,c(u_0)\}$, otherwise we could obtain a contradiction by recoloring $u_1$ with something other than those 4 colors.
    
    If $c(u_0)\in\{r,y\}$, then $L(u_1)=\{b,r,y\}$.  If $c(u_0)=y$, obtain a contradiction by defining a new coloring $c'$ in which $c'(u_1)=r$, $c'(u_2)=p$, and $c'=c$ for all other vertices.  If $c(u_0)=r$, obtain a contradiction by defining a new coloring $c''$ in which $c''(u_1)=y$, $c''(v_1)=b$, $c''(v_2)=p$, and $c''=c$ for all other vertices.
    
    If $c(u_0)\not\in\{r,y\}$, there are 3 possible lists for $u_1$.  If $L(u_1)=\{b,r,y\}$ or $L(u_1)=\{b,y,c(u_0)\}$, obtain a contradiction by defining a new coloring $c'$ in which $c'(u_1)=y$, $c'(v_1)=b$, $c'(v_2)=p$, and $c'=c$ for all other vertices.  If $L(u_1)=\{b,r,c(u_0)\}$, obtain a contradiction by defining a new coloring $c''$ in which $c''(u_1)=r$, $c''(u_2)=p$, and $c''=c$ for all other vertices.
    
    In every case and sub-case we are able to obtain a contradiction.  Thus, it must be that $c$ does not contain four blue vertices in four consecutive rungs.
\end{proof}

A consequence of Lemma~\ref{lem:fourconsecutiverungs} is that in the lex-min coloring $c$, the blue vertices of $\Pi_n$, $n\geq 6$, must be arranged into blocks of 1, 2, or 3 consecutive rungs.  That is, we can cover the vertices and edges of $\Pi_n$ with disjoint copies of the blocks shown in Figure~\ref{fig:blocks}.  The vertically-oriented edges in these diagrams correspond to rungs of $\Pi_n$, and a block may be reflected vertically.  Additionally, the blue vertices in these figures must be blue, the blank vertices must \emph{not} be blue, and each block must be properly colored. 
    \begin{figure}[!htb]
    \begin{center}
    \begin{minipage}{0.3\textwidth}\begin{center}
    \includegraphics[height=.75in]{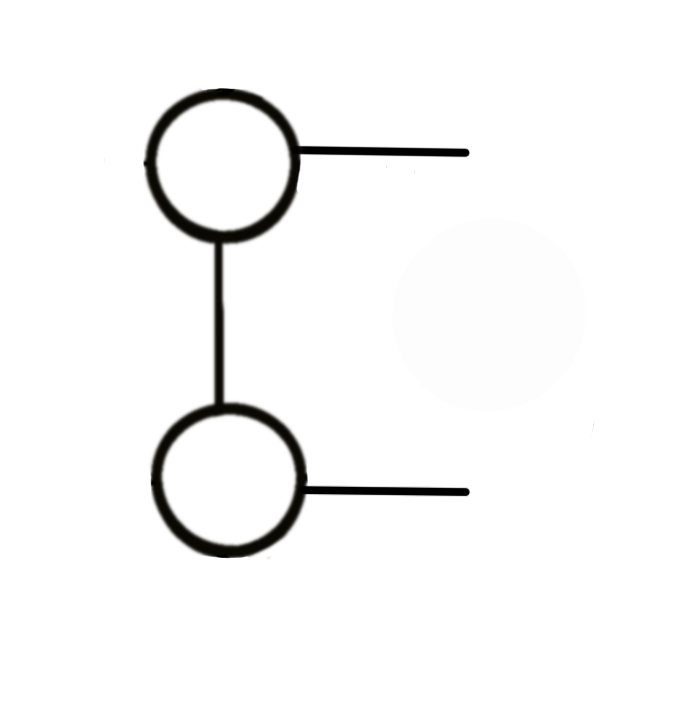}
    \captionof*{figure}{Block $B_0$}\end{center}
    \end{minipage}\hspace{0.1\textwidth}
    \begin{minipage}{0.3\textwidth}\begin{center}
    \includegraphics[height=.75in]{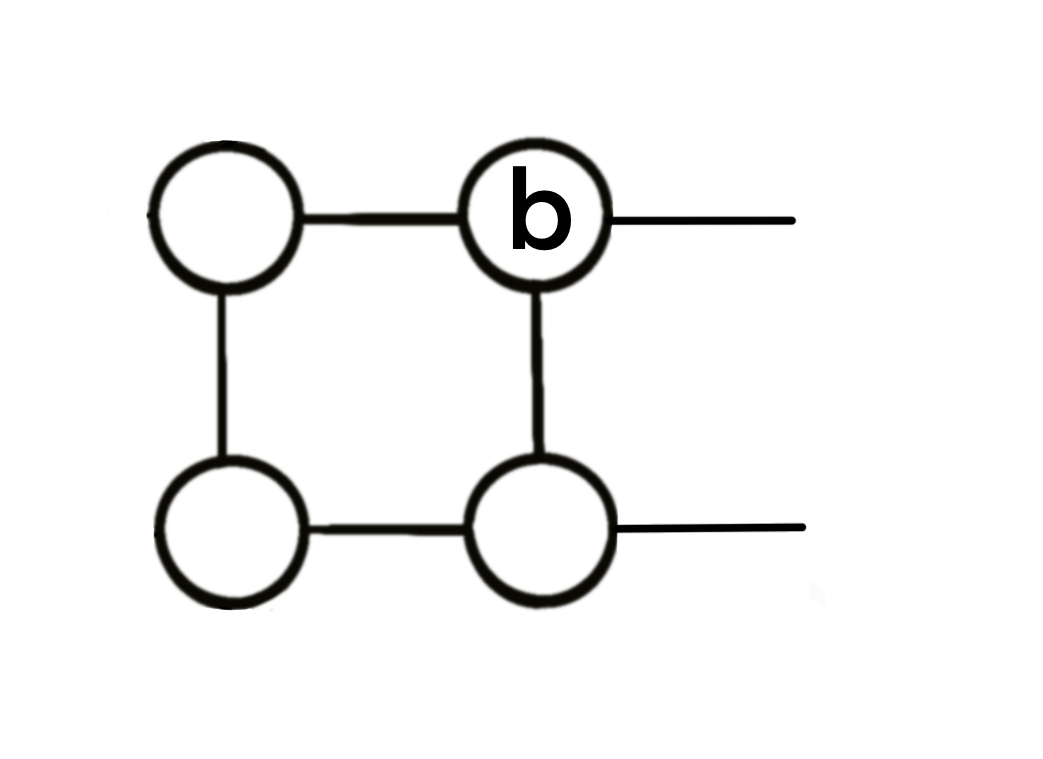} 
    \captionof*{figure}{Block $B_1$}\end{center}
    \end{minipage}
    \begin{minipage}{0.3\textwidth}\begin{center}
    \includegraphics[height=.75in]{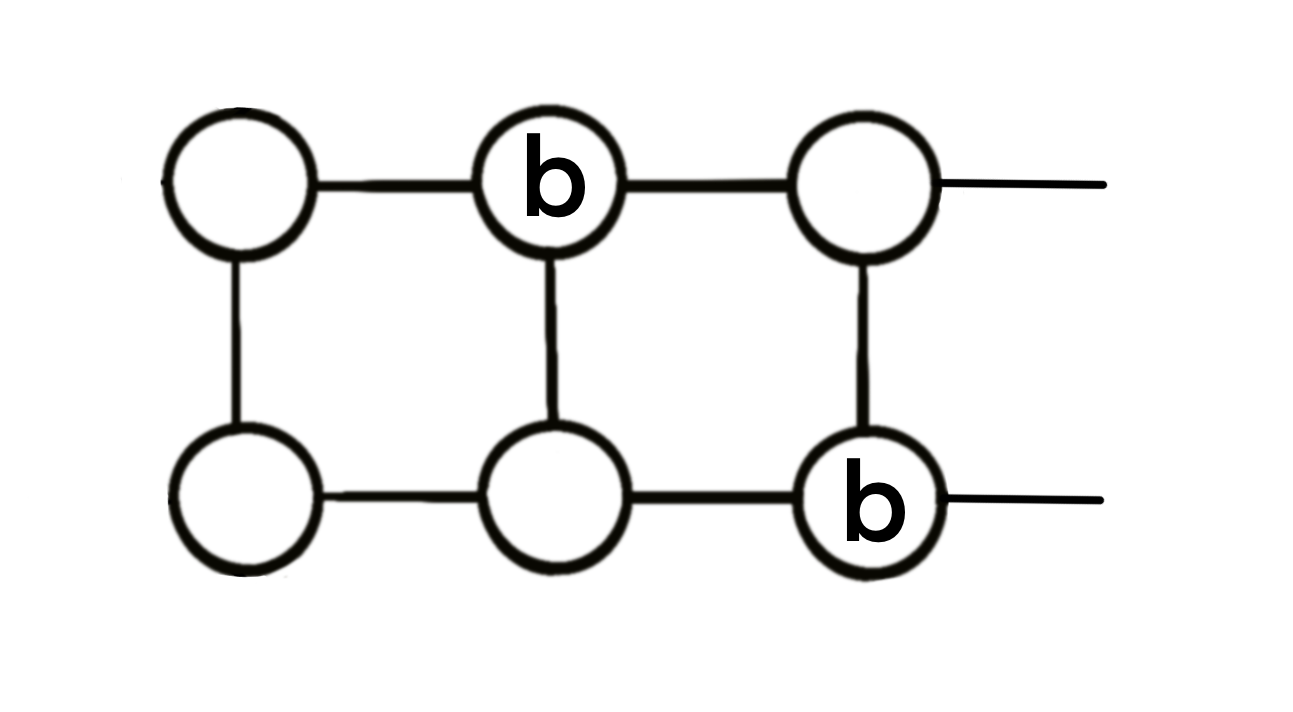} 
    \captionof*{figure}{Block $B_2$}\end{center}  
    \end{minipage}\hspace{0.1\textwidth}
    \begin{minipage}{0.3\textwidth}\begin{center}
    \includegraphics[height=.75in]{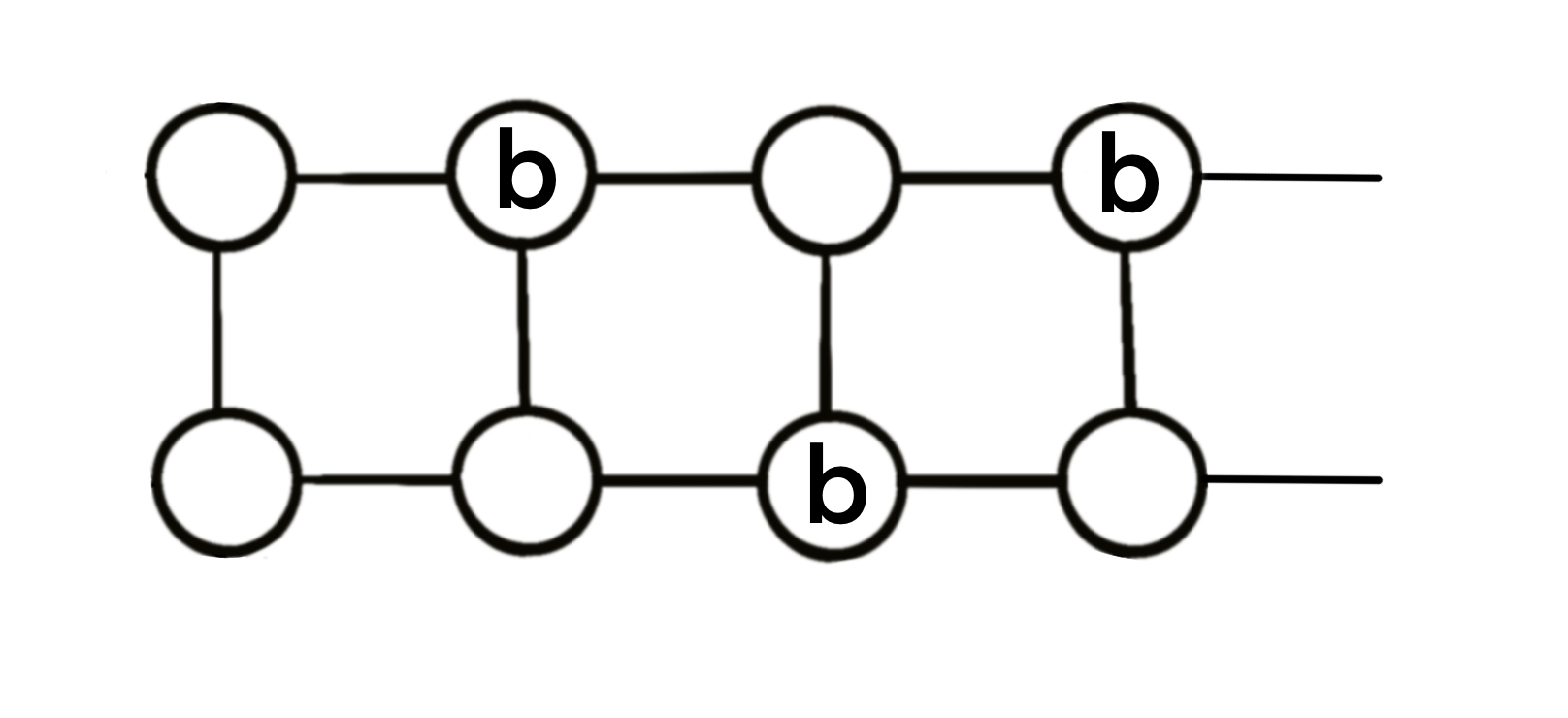} 
    \captionof*{figure}{Block $B_3$}\end{center} 
    \end{minipage}
    \caption{Blocks in the decomposition of $\Pi_n$, $n\geq 6$}
    \label{fig:blocks}
    \end{center}
    \end{figure}

The block decomposition of $\Pi_n$, $n\geq 6$ will play an important role in the proof of Theorem~\ref{thm:mainresult}.  But first, in Lemma~\ref{lem:adjacent2plusand3blocks}, we show that a $B_3$ block cannot be adjacent to either a $B_2$ block or another $B_3$ block.

\begin{lemma}\label{lem:adjacent2plusand3blocks}
    If $c$ is lex-min and not $\left\lceil 2n/3\right\rceil$-bounded, then a $B_3$ block in $\Pi_n$, $n\geq 6$ cannot be adjacent to a $B_2$ block or another $B_3$ block.
\end{lemma}
\begin{proof}
    Suppose not.  Then $\Pi_n$ contains a subgraph of 7 consecutive rungs which looks like one of configurations $F_7$ and $F_8$ in Figure~\ref{fig:noadj2plusand3blocks}.  Vertices $u_3$ and $v_3$ must be non-blue and different colors, so without a loss of generality, we suppose they have colors $r$ and $y$, respectively.  Note further that the blue vertices must be blue and each blank vertex could be any non-blue color which is not already in use by its neighbors.
    \begin{figure}[ht!]
    \centering
    \begin{minipage}{0.4\textwidth}
    \includegraphics[width=\linewidth]{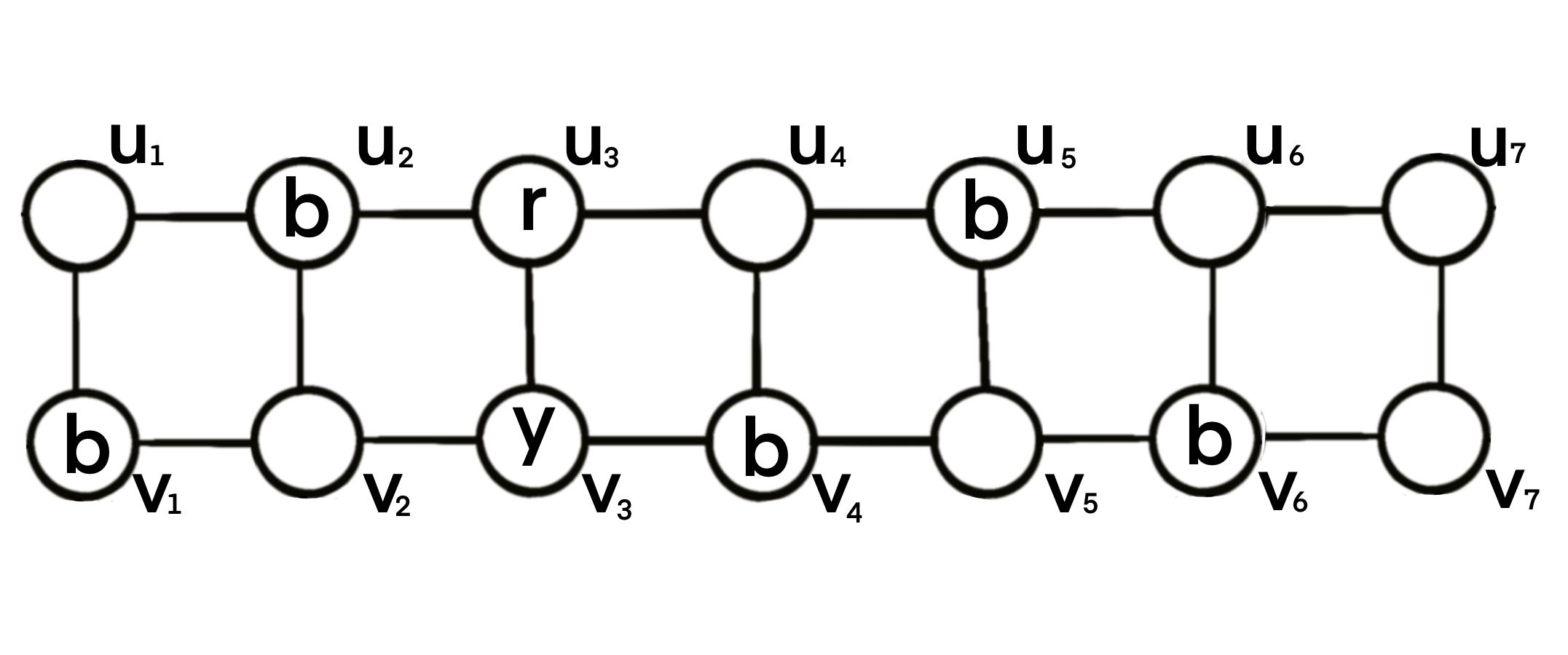}
    \captionof*{figure}{Configuration $F_7$}
    \end{minipage}\hspace{0.05\textwidth}
    \begin{minipage}{0.4\textwidth}
    \includegraphics[width=\linewidth]{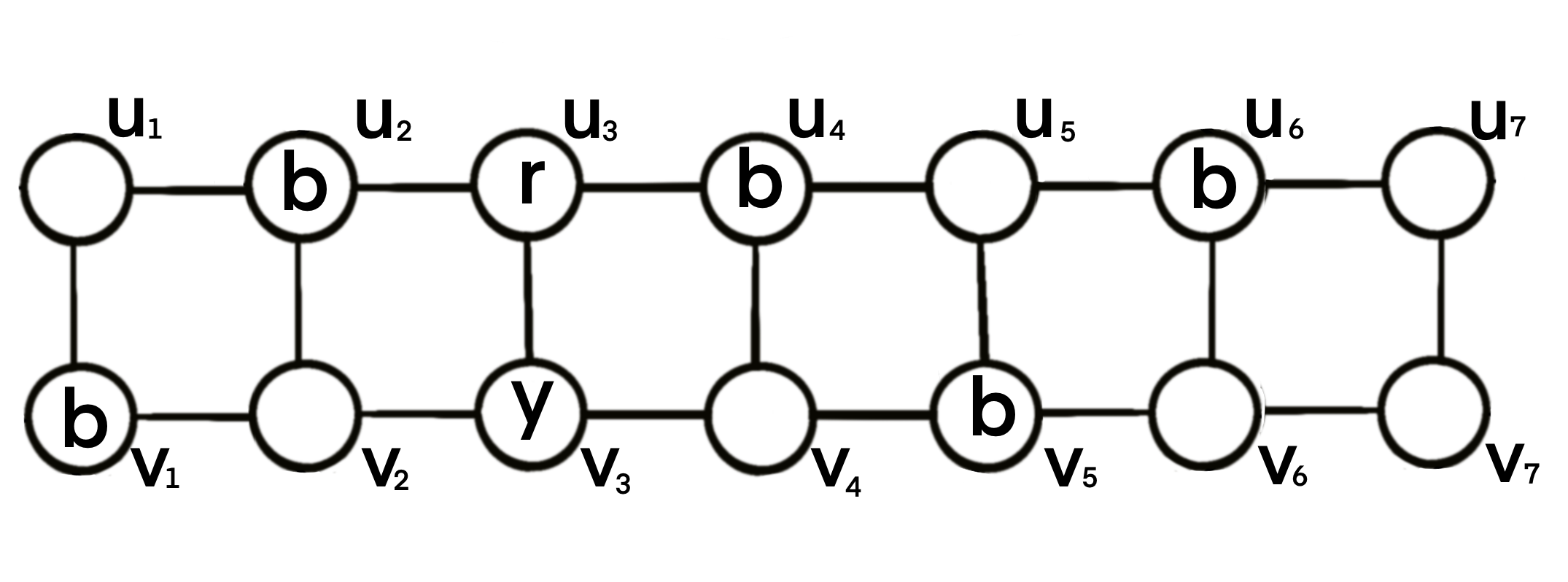}
    \captionof*{figure}{Configuration $F_8$}
    \end{minipage}
    \caption{Reducible configuations for Lemma~\ref{lem:adjacent2plusand3blocks}}
    \label{fig:noadj2plusand3blocks}
    \end{figure}
    
    In configuration $F_7$, we will first determine $L(v_4)$.  Notice that $L(v_4) \subseteq \{b,y,c(u_4),c(v_5)\}$.  If not, we could recolor $v_4$ and obtain a lexicographically smaller color word.  Further, notice that $c(u_4)\ne y$.  If it were, then $L(v_4)=\{b,y,c(v_5)\}$ so we could define a lexicographically smaller coloring $c'$ in which $c'(v_4)=c(v_5)$, $c'(v_5)\not\in\{b,c(v_5)\}$, and $c'=c$ for all other vertices.  For simplicity, say $c(u_4)=g$.  Then $L(v_4)=\{b,g,y\}$.  If not, then there exists a color $c_0$ in $L(v_4)-\{b,g,y\}$.  So we could define a lexicographically smaller coloring $c''$ in which $c''(v_4)=c_0$, $c''(v_5)\not\in\{b,c_0\}$, and $c''=c$ for all other vertices.
    
    Now we will determine $L(u_4)$, $L(u_3)$, and $L(v_3)$.  Notice that $L(u_4)=\{b,g,r\}$.  If not, we could obtain a lexicographically smaller coloring $c'$ in which $c'(v_4)=g$, $c'(u_4)\not\in\{b,g,r\}$, $c'(v_5)\not\in\{b,g\}$, and $c'=c$ for all other vertices.  Next, notice that $L(u_3)=\{b,r,y\}$.  If not, we could obtain a lexicographically smaller coloring $c''$ in which $c''(v_4)=g$, $c''(u_4)=r$, $c''(u_3)\not\in\{b,r,y\}$, $c''(v_5)\not\in\{b,g\}$, and $c''=c$ for all other vertices. Finally, notice that $L(v_3)=\{b,y,c(v_2)\}$.  If not, we could obtain a lexicographically smaller coloring $c'''$ in which $c'''(v_4)=g$, $c'''(u_4)=r$, $c'''(u_3)=y$, $c'''(v_3)\not\in\{b,y,c(v_2)\}$, $c'''(v_5)\not\in\{b,g\}$, and $c'''=c$ for all other vertices.
    
    Since we now know the color lists for $u_3$, $v_3$, $u_4$, and $v_4$, we will use them to obtain a lexicographically smaller coloring $c'$.  If $c(v_2)\ne r$, define $c'$ so that $c'(v_4)=y$, $c'(v_3)=c(v_2)$, $c'(v_2)\not\in\{b,c(v_2)\}$, and $c'=c$ for all other vertices.  If $c(v_2)=r$, define $c'$ so that $c'(v_4)=g$, $c'(u_4)=r$, $c'(u_3)=y$, $c'(v_3)=r$, $c'(v_2)\not\in\{b,r\}$, and $c'=c$ for all other vertices.  In both cases, we contradict the assumption that $c$ is a lex-min coloring.  Thus, configuration $F_7$ cannot occur in coloring $c$.
    
    In configuration $F_8$, we will first determine $L(u_4)$ and $L(u_2)$.  Notice that $L(u_4)\subseteq \{b,r,c(v_4),c(u_5)\}$.  If not, we could recolor $u_4$ and obtain a lexicographically smaller color word.  Futher, notice that $c(v_4)\ne r$.  If it were, then $L(u_4)=\{b,r,c(u_5)\}$ so we could define a lexicographically smaller coloring $c'$ in which $c'(u_4)=c(u_5)$, $c'(u_5)\not\in\{b,c(u_5)\}$, and $c'=c$ for all other vertices.  For simplicity, say $c(v_4)=g$.  Then $L(u_4)=\{b,g,r\}$.  If not, then $L(u_4)=\{b,r,c(u_5)\}$ with $c(u_5)\ne g$.  So we could define a lexicographically smaller coloring $c''$ in which $c''(u_4)=c(u_5)$, $c''(u_5)\not\in\{b,c(u_5)\}$, and $c''=c$ for all other vertices.  
    
    We will now consider three cases: $c(u_5)=r$, $c(u_5)=g$, and $c(u_5)\not\in\{g,r\}$.
    
    \textbf{Case 1}: Suppose that $c(u_5)=r$. We will determine $L(v_4)$, $L(v_3)$, and $L(u_3)$.  Notice that $L(v_4)=\{b,g,y\}$.  If not, we could obtain a lexicographically smaller coloring $c'$ in which $c'(u_4)=g$, $c'(v_4)\not\in\{b,g,y\}$, and $c'=c$ for all other vertices.  Next, notice that $L(v_3)=\{r,y,c(v_2)\}$.  If not, we could obtain a lexicographically smaller coloring $c''$ in which $c''(u_4)=g$, $c''(v_4)=y$, $c''(v_3)\not\in\{r,y,c(v_2)\}$, and $c''=c$ for all other vertices.  Finally, notice that $L(u_3)=\{b,g,r\}$.  If not, we could obtain a lexicographically smaller coloring $c'''$ in which $c'''(u_4)=g$, $c'''(v_4)=y$, $c'''(v_3)=r$, $c'''(u_3)\not\in\{b,g,r\}$, and $c'''=c$ for all other vertices.
    
    Since we now know the color lists for $u_4$, $v_4$, and $v_3$, we will use them to obtain a lexicographically smaller coloring $c'$.  Let $c'(u_4)=g$, $c'(v_4)=y$, $c'(v_3)=c(v_2)$, $c'(v_2)\not\in\{b,c(v_2)\}$, and $c'=c$ for all other vertices.  This contradicts the assumption that $c$ is a lex-min coloring.
    
    \textbf{Case 2}: Suppose $c(u_5)=g$.  We will determine $L(u_3)$ and $L(v_3)$.  Notice that $L(u_3)=\{b,r,y\}$.  If not, we could obtain a lexicographically smaller coloring $c'$ in which $c'(u_4)=r$, $c'(u_3)\not\in\{b,r,y\}$, and $c'=c$ for all other vertices.  Next, notice that $L(v_3)=\{g,y,c(v_2)\}$.  If not, we could obtain a lexicographically smaller coloring $c''$ in which $c''(u_4)=r$, $c''(u_3)=y$, $c''(v_3)\not\in\{g,y,c(v_2)\}$, and $c''=c$ for all other vertices.
    
    Since we now know the color lists for $u_4$, $u_3$, and $v_3$, we will use them to obtain a lexicographically smaller coloring $c'$.  Let $c'(u_4)=r$, $c'(u_3)=y$, $c'(v_3)=c(v_2)$, $c'(v_2)\not\in\{b,c(v_2)\}$, and $c'=c$ for all other vertices.  This contradicts the assumption that $c$ is a lex-min coloring.
    
    \textbf{Case 3}: Suppose $c(u_5)\not\in\{g,r\}$.  We can apply the same argument as Case 2 to obtain a coloring that is lexicographically smaller than $c$.  
    
    In each of the three cases, we contradict the assumption that $c$ is a lex-min coloring.  Thus, configuration $F_8$ cannot occur in coloring $c$. 
\end{proof}

We are now ready to finish proving our main result.  Theorem~\ref{thm:mainresult} argues that every lex-min 3-list-coloring of $\Pi_n$, $n\geq 6$ is $\left\lceil 2n/3\right\rceil$-bounded.  Together with Theorem~\ref{thm:3choose} and Lemmas~\ref{lem:35prisms} and~\ref{lem:4prisms}, this final proof confirms that our main result, Theorem~\ref{thm:mainthm}, is true, and $\Pi_n$, $n\geq 3$, is equitably 3-choosable.

\begin{thm}\label{thm:mainresult}
    If $c$ is lex-min $L$-coloring of $\Pi_n$, $n\geq 6$, then $c$ is $\left\lceil 2n/3\right\rceil$-bounded.
\end{thm}
\begin{proof}
    Suppose not.  Define a charge function $chg$ on the vertices and faces of $\Pi_n$, $n\geq 6$, as follows:
    \begin{itemize}
        \item[] For every vertex $v$, $chg(v)=1-|B(v)|$ where $B(v)$ is the set of neighbors of $v$ in $\Pi_n$ which are colored blue under $c$.
        \item[] For every 4-face $f$, $chg(f)=\frac{4}{3}-|B(f)|$ where $B(f)$ is the set of vertices on face $f$ which are colored blue under $c$.
        \item[] For every $n$-face $f$, $chg(f)=0$.
    \end{itemize}
    Let $V$ denote the set of vertices of $\Pi_n$ and $F$ denote the set of faces of $\Pi_n$.  Note that $F$ consists of two $n$-faces and $n$ 4-faces.  Since $|\text{Blue}|>\left\lceil 2n/3\right\rceil$, the total charge of $\Pi_n$ is
    \begin{align*}
        chg(\Pi_n) &= \sum_{f\in F} chg(f) + \sum_{v\in V} chg(v) \\
                  &= 2(0) + \frac{4n}{3} - 2|\text{Blue}| + 2n - 3|\text{Blue}| \\
                  &= 10n/3-5|\text{Blue}| \\
                  &< 10n/3-5\left\lceil 2n/3\right\rceil \\
                  &\leq 10n/3-5(2n/3) \\
                  &=0.
    \end{align*}
    That is, when $c$ is not $\left\lceil 2n/3\right\rceil$-bounded, the total charge of $\Pi_n$ is negative.
    
    Recall that we  can cover $\Pi_n$ with disjoint copies of the blocks $B_0$, $B_1$, $B_2$, and $B_3$ pictured in Figure~\ref{fig:blocks}.  As suggested by the figure, we introduce the convention that each block includes the 4-face to its right, and compute the charge of each block by summing the charges of its vertices and faces.  To begin, the blocks have the following charges:
    \begin{itemize}
        \item[] $chg(B_0)=10/3$ if the block to its left is $B_0$ and $chg(B_0)=7/3$ otherwise.
        \item[] $chg(B_1)=8/3$ if the block to its left is $B_0$ and $chg(B_1)=5/3$ otherwise.
        \item[] $chg(B_2)=1$ if the block to its left is $B_0$ and $chg(B_2)=0$ otherwise.
        \item[] $chg(B_3)=-2/3$ if the block to its left is $B_0$ and $chg(B_3)=-5/3$ otherwise.
    \end{itemize}
    We introduce the following discharging rules:
    \begin{itemize}
        \item[] \textbf{Rule 1}: Every $B_0$ steals $+1$ charge from the block to its right.
        \item[] \textbf{Rule 2}: Every $B_3$ steals $+5/3$ charge from the block to its right.
    \end{itemize}
    After applying Rule 1, blocks of each type all have the same charge.  In particular, $chg(B_0)=10/3$, $chg(B_1)=5/3$, $chg(B_2)=0$, and $chg(B_3)=-5/3$.  After applying Rule 2, $chg(B_3)=0$ for every $B_3$ block.  Further, by Lemma~\ref{lem:adjacent2plusand3blocks}, we know that the block to the right of any $B_3$ cannot be a $B_2$ or a $B_3$.  Thus, we end with the following charges:
    \begin{itemize}
        \item[] $chg(B_0)=5/3$ if the block to its left is $B_3$ and $chg(B_0)=10/3$ otherwise.
        \item[] $chg(B_1)=0$ if the block to its left is $B_3$ and $chg(B_1)=5/3$ otherwise.
        \item[] $chg(B_2)= chg(B_3)=0$ for all $B_2$ and $B_3$ blocks.
    \end{itemize}
    Since $chg(B_i)\geq 0$ for all $i\in\{0,1,2,3\}$, and $chg(\Pi_n)$ is equal to the sum of the charges of the blocks, we may conclude that $chg(\Pi_n)\geq 0$.  This contradicts our earlier conclusion that $chg(\Pi_n)<0$, so it must be that our assumption about $c$ is wrong.  Instead, every lex-min coloring $c$ of $\Pi_n$, $n\geq 6$, must be $\left\lceil 2n/3\right\rceil$-bounded.
\end{proof}


\bibliography{references}

\providecommand{\bysame}{\leavevmode\hbox to3em{\hrulefill}\thinspace}
\providecommand{\MR}{\relax\ifhmode\unskip\space\fi MR }
\providecommand{\MRhref}[2]{%
  \href{http://www.ams.org/mathscinet-getitem?mr=#1}{#2}
}
\providecommand{\href}[2]{#2}
\begin{thebibliography}{10}

\bibitem{alontarsi}
N.~Alon and M.~Tarsi, \emph{Colorings and orientations of graphs},
  Combinatorica \textbf{12} (1992), no.~2, 125--134. \MR{1179249}

\bibitem{ChenLihWu}
B.~L. Chen, K.-W. Lih, and P.-L. Wu, \emph{Equitable coloring and the maximum
  degree}, European J. Combin. \textbf{15} (1994), no.~5, 443--447.
  \MR{1292955}

\bibitem{Diestel}
R.~Diestel, \emph{Graph theory}, 5th ed., Graduate Texts in Mathematics, vol.
  173, Springer, Berlin, 2017. \MR{3644391}

\bibitem{DZ}
A.~Dong and X.~Zhang, \emph{Equitable coloring and equitable choosability of
  graphs with small maximum average degree}, Discuss. Math. Graph Theory
  \textbf{38} (2018), no.~3, 829--839. \MR{3811968}

\bibitem{ERT}
P.~Erd\H{o}s, A.~L. Rubin, and H.~Taylor, \emph{Choosability in graphs},
  Proceedings of the {W}est {C}oast {C}onference on {C}ombinatorics, {G}raph
  {T}heory and {C}omputing ({H}umboldt {S}tate {U}niv., {A}rcata, {C}alif.,
  1979), Congress. Numer., XXVI, Utilitas Math., Winnipeg, Man., 1980,
  pp.~125--157. \MR{593902}

\bibitem{grotzsch}
H.~Gr\"{o}tzsch, \emph{Zur {T}heorie der diskreten {G}ebilde. {VII}. {E}in
  {D}reifarbensatz f\"{u}r dreikreisfreie {N}etze auf der {K}ugel}, Wiss. Z.
  Martin-Luther-Univ. Halle-Wittenberg Math.-Natur. Reihe \textbf{8} (1958/59),
  109--120. \MR{116320}

\bibitem{KPW}
A.~V. Kostochka, M.~J. Pelsmajer, and D.~B. West, \emph{A list analogue of
  equitable coloring}, J. Graph Theory \textbf{44} (2003), no.~3, 166--177.
  \MR{2012800}

\bibitem{pelsmajer}
M.~J. Pelsmajer, \emph{Equitable list-coloring for graphs of maximum degree 3},
  J. Graph Theory \textbf{47} (2004), no.~1, 1--8. \MR{2079286}

\bibitem{vizing}
V.~G. Vizing, \emph{Coloring the vertices of a graph in prescribed colors},
  Diskret. Analiz (1976), no.~29, Metody Diskret. Anal. v Teorii Kodov i Shem,
  3--10, 101. \MR{498216}

\bibitem{WangLih}
W.-F. Wang and K.-W. Lih, \emph{Equitable list coloring of graphs}, Taiwanese
  J. Math. \textbf{8} (2004), no.~4, 747--759. \MR{2105563}

\bibitem{ZhuBu}
J.~Zhu and Y.~Bu, \emph{Equitable and equitable list colorings of graphs},
  Theoret. Comput. Sci. \textbf{411} (2010), no.~43, 3873--3876. \MR{2779758}

\end{thebibliography}

\end{document}